\documentclass[graybox,a4wide]{svmult}

\usepackage{makeidx}         % allows index generation
\usepackage{graphicx}        % standard LaTeX graphics tool
\usepackage{multicol}        % used for the two-column index
\usepackage[bottom]{footmisc}% places footnotes at page bottom

\usepackage{bm}
\usepackage{amsmath}
\usepackage{amssymb}
\usepackage{mathrsfs}
\usepackage{algorithm,algorithmic}

\usepackage{newtxtext}       % 
\usepackage{newtxmath}       % selects Times Roman as basic font

\makeindex             % used for the subject index
\begin{document}

\title*{Balanced truncation model reduction for \\3D linear magneto-quasistatic field problems }
\author{Johanna Kerler-Back and Tatjana Stykel}
\institute{Johanna Kerler-Back \at 
	Institut f\"ur Mathematik, Universit\"at Augsburg, Universit\"atsstra{\ss}e 14, 86159 Augsburg, Germany \email{kerler@math.uni-augsburg.de}
\and Tatjana Stykel \at 
    Institut f\"ur Mathematik, Universit\"at Augsburg, Universit\"atsstra{\ss}e 14, 86159 Augsburg, Germany \email{stykel@math.uni-augsburg.de}}
\maketitle

\abstract*{We consider linear magneto-quasistatic field equations which arise in simulation 
	of low-frequency electromagnetic devices coupled to electrical circuits.
	A finite element discretization of such equations on 3D domains leads to a singular system 
	of differential-algebraic equations. First, we study the structural properties of such a system 
	and present a new regularization approach based on projecting out the singular state components. 
	Furthermore, we present a Lyapunov-based balanced truncation model reduction method which preserves stability 
	and passivity. By making use of the underlying structure of the problem, we develop an efficient model reduction algorithm. 
	Numerical experiments demonstrate its performance on a test example.}

\abstract{We consider linear magneto-quasistatic field equations which arise in simulation 
	of low-frequency electromagnetic devices coupled to electrical circuits.
	A~finite element discretization of such equations on 3D domains leads to a~singular system 
	of differential-algebraic equations. First, we study the structural properties of such a~system 
	and present a~new regularization approach based on projecting out the singular state components. 
	Furthermore, we present a~Lyapunov-based balanced truncation model reduction method which preserves stability 
	and passivity. By ma\-king use of the underlying structure of the problem, we develop an efficient model reduction algorithm. 
	Numerical experiments demonstrate its performance on a~test example.}

\section{Introduction}
Nowadays, integrated circuits play an~increasingly important role. Modelling of elec\-tro\-mag\-ne\-tic effects
in high-frequency and high-speed electronic systems leads to coupled field-circuit models of high complexity.
The development of efficient, fast and accurate simulation tools for such models is of great importance
in the computer-aided design of electromagnetic structures offering significant savings in production cost and time.

In this paper, we consider model order reduction of linear magneto-quasistatic (MQS) systems obtained from Maxwell's equations by assuming that the contribution of displacement current is negligible compared to the conductive currents. Such systems are commonly used for modeling of low-frequency electromagnetic devices like transformers, induction sensors and generators. 
Due to the presence of non-conducting subdomains, MQS models take form of partial differential-algebraic equations whose dynamics are restricted to a manifold described by algebraic constraints. 
A spatial discretization of MQS systems using the finite integration technique (FIT) \cite{Wei77} or the finite element method (FEM) \cite{Boss1998,Monk03,Nede1980} leads to 
differential-algebraic equations (DAEs) which are singular in the 3D case. The structural analysis and numerical treatment of singular DAEs is facing serious challenges due to the fact that the inhomogeneity has to satisfy some restricted conditions
to guarantee the existence of solutions and/or that the solution space is infinite-dimensional. 
To overcome these difficulties, different regularization techniques have been developed for MQS systems \cite{Boss2001, CleSchpsGerBar2011, CleWei2002, Hipt2000}. Here, we propose a~new regularization approach which is based on a~special state space transformation and withdrawal of overdetermined state components and redundant equations. 

Furthermore, we exploit the special block structure of the regularized MQS system to 
determine the deflating subspaces of the underlying matrix pencil corresponding to zero and infinite eigenvalues. This makes it possible to extend the balanced truncation model reduction method to 
3D MQS problems. Similarly to \cite{KS17, ReiSty2009}, our approach relies on projected Lyapunov equations and preserves passivity in a~reduced-order model. It should be noted that the balanced truncation method presented in \cite{KS17} for 2D and 3D gauging-regularized MQS systems cannot be applied to the regularized 
system obtained here, since it is stable, but not asymptotically stable. To get rid of this problem, 
we proceed as in \cite{ReiSty2009} and project out state components corresponding not only to the eigenvalue at infinity, but also to zero eigenvalues. Our method is based on computing certain subspaces of incidence matrices related to the FEM discretization which can be determined by using efficient graph-theoretic algorithms as developed in \cite{Ipac2013}. 

\section{Model Problem}

We consider a system of MQS equations in vector potential formulation given by
\begin{equation}
\arraycolsep=2pt
\begin{array}{rclrl}
\displaystyle{\sigma\frac{\partial \mathbf{A}}{\partial t} + \nabla \times \nu\,
\nabla \times \mathbf{A}} & = & \chi\,  \iota 
& \quad \mbox{in}& \Omega\times (0,T), \\
\mathbf{A}\times n_o & = & 0 & \quad \mbox{on}& \partial \Omega\times (0,T), \\[1mm]
\mathbf{A}(\,\cdot\,, 0)& = & \mathbf{A}_0 & \quad \mbox{in}& \Omega, \\
\displaystyle{\int_\Omega \chi^T\frac{\partial \mathbf{A} }{\partial t}\, {\rm d}\xi + R\, \iota} & = & u & \quad \mbox{in} & (0,T), 
\end{array}
\label{eq:MQSmodel}
\end{equation}
where $\mathbf{A}:\Omega\times (0,T)\to\mathbb{R}^3$ is the magnetic vector potential,
 $\chi:\Omega\to\mathbb{R}^{3\times\, m}$ is a~divergence-free winding function, 
 $\iota:(0,T)\to\mathbb{R}^m$ and $u:(0,T)\to\mathbb{R}^m$ are the electrical current and voltage through the stranded conductors with $m$ terminals. Here, $\Omega\subset\mathbb{R}^3$ is a~bounded simply connected domain with a~Lipschitz 
 boundary $\partial\Omega$, and $n_o$ is an~outer unit normal vector to $\partial \Omega$. 
 The MQS system \eqref{eq:MQSmodel} is obtained from Maxwell's equations by neglecting the contribution of the displacement currents. It~is used to study the dynamical behavior of magnetic fields in low-frequency applications \cite{HauM89,RodV10}. 
 The integ\-ral equation in \eqref{eq:MQSmodel} with a symmetric, positive definite resistance matrix $R\in\mathbb{R}^{m\times m}$ results from Faraday's induction law. This equation describes the 
 coupling the electromagnetic devices to an external circuit \cite{SchpsGerWei2013}. Thereby, the voltage $u$ is assumed to be given and the current $\iota$ has to be determined. In this case, the MQS system \eqref{eq:MQSmodel} can be considered as a~control system with the input $u$, the state $[\mathbf{A}^T,\iota^T]^T$ and the output $y=\iota$.  
 
 We assume that the domain $\Omega$ is composed of the conducting and non-conducting subdomains $\Omega_1$ and $\Omega_2$, respectively, such that $\overline{\Omega}=\overline{\Omega}_1\cup\overline{\Omega}_2$,  \mbox{$\Omega_1\cap \Omega_2=\emptyset$}  and \mbox{$\overline{\Omega}_1\subset \Omega$}. Furthermore, we restrict ourselves to linear isotropic media implying that the electrical conductivity $\sigma$ and the magnetic reluctivity $\nu$ are scalar functions of the spatial variable only. The electrical conductivity $\sigma:\Omega\to\mathbb{R}$ is given by 
 $$
 \sigma(\xi)=\begin{cases} 
 \;\sigma_1 & \mbox{in}\; \Omega_1, \\ 
 \;0        & \mbox{in}\; \Omega_2
 \end{cases}
$$
 with some constant $\sigma_1>0$, whereas the magnetic reluctivity $\nu:\Omega\to\mathbb{R}$ is bounded, measurable and uniformly positive such that $\nu(\xi)\geq \nu_0>0$ for a.e. in $\Omega$. Note that since $\sigma$ vanishes on the non-conducting subdomain $\Omega_2$, the initial condition $\mathbf{A}_0$ can only be prescribed in the conducting subdomain $\Omega_1$. Finally, for the winding function $\chi=[\chi_1,\ldots,\chi_m]$, we assume that 
\begin{align}
&\overline{\mbox{supp}(\chi_j)}\subset\Omega_2, \qquad\enskip j=1,\ldots, m,\label{eq:chi1} \\
&\mbox{supp}(\chi_i)\cap \mbox{supp}(\chi_j)=\emptyset\quad \mbox{for}\;i\neq j. \label{eq:chi2}
\end{align}
These conditions mean that the conductor terminals are located in $\Omega_2$ and they do not intersect \cite{SchpsGerWei2013}.
 
\subsection{FEM Discretization}

First, we present a~weak formulation for the MQS system \eqref{eq:MQSmodel}. For this purpose, we multiply the first equation in \eqref{eq:MQSmodel} with a~test function $\phi\in H_0(\mbox{curl}, \Omega)$ and integrate over the domain $\Omega$. Using Green's formula, we obtain the variational equation 
\begin{equation}\label{eq:MQSweak3D}
 \begin{aligned}
 \frac{\partial}{\partial t} \int_{\Omega} \sigma \mathbf{A}\cdot \phi\, {\rm d} \xi + \int_{\Omega} \nu\, (\nabla \times \mathbf{A})\cdot(\nabla \times \phi)\, {\rm d}\xi & =  \int_{\Omega} (\chi \iota)\cdot\phi\, {\rm d}\xi,\\
 \frac{\partial}{\partial t}\int_{\Omega} \chi^T \mathbf{A}\, {\rm d} \xi + R\, \iota &= u,\\
\mathbf{A}(\cdot, 0)&=\mathbf{A}_0.
 \end{aligned}
\end{equation}
The existence, uniqueness and regularity results for this equation can be found in~\cite{NicTroe2013}.

For a~spatial discretization of \eqref{eq:MQSweak3D}, we use N\'ed\'elec edge and face ele\-ments as introduced in \cite{Nede1980}. Let $\mathcal{T}_h(\Omega)$ be a~regular simplicial triangulation of $\Omega$, and let $n_n$, $n_e$ and $n_f$ denote the number of nodes, edges and facets, respectively. Furthermore, let $\Phi^e=[ \phi_1^e, \ldots, \phi_{n_e}^e]$ and $\Phi^f=[ \phi_1^f, \ldots, \phi_{n_f}^f]$ be the edge and face basis functions, respectively, which span the corresponding finite element spaces. They are related via
\begin{equation}
\nabla \times \Phi^e=\Phi^f C,
\label{eq:ef}
\end{equation}
where $C\in\mathbb{R}^{n_f\times n_e}$ is a~{\em discrete curl matrix} with entries
\begin{equation*}
C_{ij}=\begin{cases} 
\phantom{-}1, &\text{ if edge $j$ belongs to face $i$ and their orientations match},\\
-1, &\text{ if edge $j$ belongs to face $i$ and their orientations do not match},\\
\phantom{-}0, & \text{ if edge $j$ does not belong to face $i$},
\end{cases}
\end{equation*}
see \cite[Section 5]{Boss1998}. Substituting an~approximation to the magnetic vector potential
$$
\mathbf{A}(\xi,t)\approx \sum_{j=1}^{n_e} \alpha_j(t)\, \phi^e_j(\xi)
$$
into the variational equation \eqref{eq:MQSweak3D} and testing it with $\phi_i^e\in H_0(\mbox{curl}, \Omega)$,
we obtain a~linear DAE system 
\begin{align}\label{eq:MQSDAE3D}
\begin{bmatrix}
  M &\enskip 0\\
  X^T &\enskip 0
 \end{bmatrix}
\frac{d}{dt}
\begin{bmatrix}
 a\\ \iota
\end{bmatrix}
=
\begin{bmatrix}
  -K &\enskip\; X \\
  0&\enskip -R
 \end{bmatrix}
\begin{bmatrix}
 a\\ \iota
\end{bmatrix}+
\begin{bmatrix} 0\\I \end{bmatrix} u,
\end{align}
where $a=\begin{bmatrix}
\alpha_1, \ldots,  \alpha_{n_e} \end{bmatrix}^T$ and the conductivity matrix
$M\in\mathbb{R}^{n_e\times \,n_e}$, the curl-curl matrix $K\in\mathbb{R}^{n_e\times \,n_e}$ and the coupling matrix $X\in\mathbb{R}^{n_e\times \,m}$ have entries
\begin{equation}\label{eq:MQSDAE3Dmat}
\arraycolsep=2pt
\begin{array}{rlr}
M_{ij}&\displaystyle{=\int_{\Omega} \sigma\, \phi_j^e \cdot \phi_i^e\, {\rm d}\xi,} \qquad\qquad\qquad\enskip i,&\!j=1,\ldots, n_e,\\[4mm] 
K_{ij}&\displaystyle{= \int_{\Omega} \nu\,  (\nabla \times \phi_j^e)\cdot(\nabla \times \phi_i^e)\, {\rm d}\xi,} \qquad i,&\!j=1,\ldots, n_e,\\[4mm]
X_{ij}&\displaystyle{=\int_{\Omega} \chi_j \cdot\phi_i^e\, {\rm d}\xi,}\qquad\; i=1,\ldots, n_e, &j=1,\ldots, m.
\end{array}
\end{equation}
Note that the matrices $M$ and $K$ are symmetric, positive semidefinite. Using the relation \eqref{eq:ef}, we can rewrite the matrix $K$ as
\begin{equation*}
\begin{array}{rcl}
 K&=&\displaystyle{\int_{\Omega} \nu\, (\nabla \times\Phi^e)^T (\nabla \times \Phi^e)\, {\rm d}\xi} =\displaystyle{\int_{\Omega} \nu\, C^T(\Phi^f)^T\Phi^f C\, {\rm d}\xi} = C^T \! M_\nu C,
\end{array}
\end{equation*}
where the entries of the symmetric and positive definite matrix $M_\nu$ are given by
\begin{equation*}
(M_\nu)_{ij}=\int_{\Omega} \nu \,\phi_j^f \cdot \phi_i^f\, {\rm d}\xi,\qquad i,j=1,\ldots, n_f.
\end{equation*}
The coupling matrix $X$ can also be represented in a factored form using the discrete curl matrix $C$. This can be achieved by taking into account the divergence free property of the winding function $\chi$, which implies $\chi=\nabla\times\gamma$ for a certain matrix-valued function 
$$
\gamma=[\gamma_1,\ldots,\gamma_m]:\Omega\rightarrow \mathbb{R}^{3\times m}.
$$ 
Using the cross product rule, Gauss's theorem as well as relations \eqref{eq:ef} and $\phi_i^e\times n_o=0$ on $\partial\Omega$, we obtain
\begin{align*}
X_{ij}&=\int_{\Omega} (\nabla \times \gamma_j)\cdot \phi_i^e \, {\rm d}\xi 
=\int_{\Omega} \nabla\cdot(\gamma_j\times \phi_i^e)\,{\rm d}\xi +
 \int_{\Omega}\gamma_j\cdot(\nabla\times \phi_i^e)\, {\rm d}\xi\\
&= \int_{\partial\Omega} (\gamma_j\times \phi_i^e)\cdot n_o\,{\rm d}s +
   \int_{\Omega} \gamma_j\cdot \sum_{k=1}^{n_f} C_{ki} \phi_k^f \,{\rm d}\xi \\
&= \int_{\partial\Omega} \gamma_j\cdot(\phi_i^e\times  n_o)\,{\rm d}s +
  \sum_{k=1}^{n_f} C_{ki} \int_{\Omega} \gamma_j\cdot  \phi_k^f \,{\rm d}\xi =\sum_{k=1}^{n_f} C_{ki} \int_{\Omega} \gamma_j\cdot  \phi_k^f \,{\rm d}\xi.
\end{align*}
Then the matrix $X$ can be written as $X=C^T\mathit{\Upsilon}$, 
where the entries of $\mathit{\Upsilon}\in \mathbb{R}^{n_f\times m}$ are given by
\begin{equation*}
\mathit{\Upsilon}_{kj}=\int_{\Omega} \gamma_j\cdot \phi_k^f\, d\xi,\qquad k=1,\ldots, n_f,\, j=1,\ldots,m.
\end{equation*}
Note that due to \eqref{eq:chi2}, the matrix $X$ has full column rank. This immediately implies that $\mathit{\Upsilon}$ is also of full column rank.

\section{Properties of the FEM Model} 
\label{sec:properties}

In this section, we study the structural and physical properties of the FEM model~\eqref{eq:MQSDAE3D}.
We start with reordering the state vector $a=[a_1^T,\, a_2^T]^T$ with $a_1\in\mathbb{R}^{n_1}$ and  $a_2\in\mathbb{R}^{n_2}$ accordingly to the conducting and non-conducting subdomains 
$\Omega_1$ and $\Omega_2$. Then the matrices $M$, $K$, $X$ and $C$ can be partitioned
into blocks as
$$
M=\begin{bmatrix} M_{11} &\enskip 0 \\ 0 &\enskip 0 \end{bmatrix}, \qquad
K=\begin{bmatrix} K_{11} & K_{12} \\ K_{21} & K_{22} \end{bmatrix},\qquad 
X=\begin{bmatrix} X_1 \\ X_2 \end{bmatrix}, \qquad 
C=\begin{bmatrix} C_1, \, C_2 \end{bmatrix},
$$
where 
$M_{11}\in\mathbb{R}^{n_1\times\, n_1}$ is symmetric, positive definite,  
$K_{11}\in\mathbb{R}^{n_1\times\, n_1}$, 
$K_{22}\in\mathbb{R}^{n_2\times\, n_2}$, 
\mbox{$K_{21}=K_{12}^T\in\mathbb{R}^{n_2\times\, n_1}$}, 
$X_1\in\mathbb{R}^{n_1\times\, m}$, 
$X_2\in\mathbb{R}^{n_2\times\, m}$,
$C_1\in\mathbb{R}^{n_f\times\, n_1}$, and 
$C_2\in\mathbb{R}^{n_f\times\, n_2}$.
Note that conditions \eqref{eq:chi1} and \eqref{eq:chi2} imply that $X_1=0$ and $X_2$ has full column rank. In what follows, however, we consider for completeness a~general block $X_1$. Solving the second equation in \eqref{eq:MQSDAE3D} for 
$\iota=-R^{-1}X^T\dot{a}+R^{-1}u$
and inserting this vector into the first equation in \eqref{eq:MQSDAE3D} yields the DAE control system 
\begin{equation}\label{eq:MQSDAE3Dnocoupling}
\arraycolsep=2pt
\begin{array}{rcl}
E\dot{a}&=&-Ka +Bu,\\
y&=&-B^T\! \dot{a}+R^{-1}u,
\end{array}
\end{equation}
with the matrices
\begin{equation}\label{eq:MQS3Dmat}
\arraycolsep=2pt
\begin{array}{rcl}
E&=&\begin{bmatrix}
M_{11}^{}\!+\!X_1^{} R^{-1} X_1^T & X_1^{} R^{-1} X_2^T\\[2mm]
X_2^{} R^{-1} X_1^T & X_2^{} R^{-1} X_2^T  
\end{bmatrix} 
= \begin{bmatrix} I &\; C_1^T\mathit{\Upsilon} \\[2mm] 0 &\; C_2^T\mathit{\Upsilon} \end{bmatrix}\!
\begin{bmatrix} M_{11} & 0\\[2mm] 0 & R^{-1} \end{bmatrix}\!
\begin{bmatrix} I & 0 \\[2mm] \mathit{\Upsilon}^T\! C_1 & \mathit{\Upsilon}^T\! C_2 \end{bmatrix},\\[6mm]
K&=& \begin{bmatrix}
C_1^T M_\nu C_1^{} &\enskip C_1^T M_\nu C_2^{}\\[2mm]
C_2^T M_\nu C_1^{} &\enskip C_2^T M_\nu C_2^{}
\end{bmatrix}, \qquad
B = \begin{bmatrix} X_1\\[2mm] X_2 \end{bmatrix} R^{-1} =
\begin{bmatrix} C_1^T \mathit{\Upsilon} \\[2mm] C_2^T \mathit{\Upsilon} \end{bmatrix} R^{-1}.
\end{array}
\end{equation}
Using the block structure of the matrices $E$ and $K$, we can determine their common kernel. 

\begin{theorem}
	Assume that $M_{11}$, $R$ and $M_\nu$ are symmetric and positive definite. Let the columns of
	$Y_{C_2}\in\mathbb{R}^{n_2\times\, k_2}$ form a~basis of $\ker(C_2)$.
	Then $\ker(E)\cap\ker(K)$ is spanned by columns of the matrix $\begin{bmatrix} 0,\, Y_{C_2}^T\end{bmatrix}^T$.
\end{theorem}

\begin{proof}
	Assume that $w=\begin{bmatrix} w_1^T,\, w_2^T\end{bmatrix}^T\in \ker(E)\cap\ker(K)$.
	Then due to the positive definiteness of $M_{11}$ and $R$, it follows from $w^T E w=0$ with $E$ as in \eqref{eq:MQS3Dmat} that
	%\begin{subequations}
	$$
	\begin{bmatrix}
	I &\enskip 0\\
	\mathit{\Upsilon}^T\! C_1&\enskip \mathit{\Upsilon}^T\! C_2
	\end{bmatrix} \begin{bmatrix} w_1 \\ w_2 \end{bmatrix} =0.
	$$
	Therefore, $w_1=0$ and $\mathit{\Upsilon}^TC_2w_2=0$. Moreover, using the positive definiteness of $M_\nu$, we get from 
	$w^T Kw=0$ with $w_1=0$ that $C_2 w_2=0$. This means that \mbox{$w_2\in\ker(C_2)=\mbox{im}(Y_{C_2})$},
	i.e., $w_2=Y_{C_2}z$ for some vector $z$. 
	%Note that $w_2$ also fulfills equation \eqref{eq:MQS3DkerEkerA1}. 
	Thus, $w=[ 0,\, Y_{C_2}^T ]^Tz$.
	
	Conversely, assume that $w=[0,\, Y_{C_2}^T]^T\!z$ for some $z\in\mathbb{R}^{k_2}$. Then using \eqref{eq:MQS3Dmat} and \mbox{$C_2^{}Y_{C_2}=0$}, we obtain $E w=0$ and 
	$Kw=0$. Thus, \mbox{$w\in\ker(E)\cap\ker(K)$}. 
\end{proof}

It follows from this theorem that if $C_2$ has a~nontrivial kernel, then
$$
\det(\lambda E+K)=0
$$
for all $\lambda\in\mathbb{C}$ implying that the pencil $\lambda E+K$ (and also the DAE system \eqref{eq:MQSDAE3Dnocoupling}) is singular. This may cause difficulties with the existence and uniqueness 
of the solution of \eqref{eq:MQSDAE3Dnocoupling}. In the next section, we will see that the divergence-free condition of the winding function $\chi$ guarantees that \eqref{eq:MQSDAE3Dnocoupling} is solvable, but the solution is not unique. This is a~consequence of nonuniqueness of the magnetic vector potential $\mathbf{A}$ which is defined up to a~gradient of an arbitrary scalar function.

\subsection{Regularization}

Our goal is now to regularize the singular DAE system \eqref{eq:MQSDAE3Dnocoupling}. In the literature, several regu\-la\-ri\-za\-tion approaches have been proposed for semidiscretized 3D MQS systems. In the context of the FIT discretization, the grad-div regularization of MQS systems has been considered in \cite{CleSchpsGerBar2011, CleWei2002} 
which is based on a~spatial discretization of the Coulomb gauge equation $\nabla\cdot {\mathbf A}=0$.
For other regularization techniques, we refer to \cite{Boss2001, CenMan1995, Hipt2000, Mont2002}. 
 Here, we present a~new regularization method relying on a~special coordinate transformation and elimination of the over- and underdetermined parts. 

To this end, we consider a matrix 
$\hat{Y}_{C_2}\in \mathbb{R}^{n_2\times\, (n_2-k_2)}$ whose columns form a~basis of $\mbox{im}(C_2^T)$. 
Then the matrix
\begin{equation*}
T=\begin{bmatrix}
I &\enskip 0 &\enskip 0\\
0 &\enskip \hat{Y}_{C_2} &\enskip Y_{C_2}
\end{bmatrix}
\end{equation*}
is nonsingular. Multiplying the state equation in \eqref{eq:MQSDAE3Dnocoupling} from the left with $T^T$ and introducing a~new state vector
\begin{equation}\label{eq:Tinva}
\begin{bmatrix} a_1\\ a_{21}\\ a_{22} \end{bmatrix}= T^{-1} a,
\end{equation}
the system matrices of the transformed system take the form
\begin{align*}
T^T\! E T &=\begin{bmatrix}
M_{11}+ C_1^T\mathit{\Upsilon} R^{-1}\mathit{\Upsilon}^TC_1^{} &\quad\quad C_1^T\mathit{\Upsilon} R^{-1} \mathit{\Upsilon}^TC_2^{}\hat{Y}_{C_2}^{} &\enskip 0\\[1mm]
\hat{Y}_{C_2}^T C_2^T\mathit{\Upsilon} R^{-1} \mathit{\Upsilon}^T C_1^{} & \enskip\hat{Y}_{C_2}^T C_2^T\mathit{\Upsilon} R^{-1} \mathit{\Upsilon}^T C_2^{} \hat{Y}_{C_2}^{} &\enskip 0\\[1mm]
0&0&\enskip 0
\end{bmatrix},\\
T^T\! K T &=\begin{bmatrix}
\quad C_1^TM_\nu C_1^{} &\qquad C_1^TM_\nu C_2^{}\hat{Y}_{C_2} &\enskip 0\\[1mm]
\hat{Y}_{C_2}^TC_2^TM_\nu C_1^{} &\enskip \hat{Y}_{C_2}^TC_2^TM_\nu C_2^{}\hat{Y}_{C_2} &\enskip 0\\[1mm]
0&0&\enskip 0
\end{bmatrix},\quad T^T\! B = \begin{bmatrix}
\enskip C_1^T\mathit{\Upsilon} \\[1mm]
\hat{Y}_{C_2}^T C_2^T\mathit{\Upsilon} \\[1mm] 0
\end{bmatrix}R^{-1}.  
\end{align*}
This implies that the components of $a_{22}$ are actually not involved in the transformed system and,  
therefore, they can be chosen freely. Moreover, the third equation $0=0$ is trivially satisfied showing that 
system  \eqref{eq:MQSDAE3Dnocoupling} is solvable. Removing this equation, we obtain a~regular DAE system 
\begin{align}
E_r \dot{x}_r&= \enskip\; A_r  x_r+B_r\, u, \label{eq:MQS3Dreglin} \\
y &= -B_r^T \dot{x}_r+R^{-1}u, \label{eq:MQS3Dreglinout}
\end{align}
with $x_r=[a_1^T,\; a_{21}^T]^T\in\mathbb{R}^{n_r}$, $n_r=n_1+n_2-k_2$, and 
\begin{equation}
\label{eq:MQS3Dreglinmat}
E_r =F_{\sigma}M_{\sigma}F_{\sigma}^T, \qquad A_r =-F_{\nu}M_{\nu}F_{\nu}^T, \qquad
B_r =F_{\nu}\mathit{\mathit{\Upsilon}} R^{-1},
\end{equation}
where 
$$
F_\sigma=\begin{bmatrix}I & \enskip X_1 \\ 0 & \enskip\hat{Y}_{C_2}^T X_2 \end{bmatrix}=\begin{bmatrix}I & \enskip C_1^T\mathit{\Upsilon} \\ 0 & \enskip\hat{Y}_{C_2}^T C_2^T\mathit{\Upsilon} \end{bmatrix}, \quad
M_\sigma=\begin{bmatrix} M_{11} & 0 \\0 & R^{-1} \end{bmatrix}, \quad
F_\nu=\begin{bmatrix}C_1^T\\ \hat{Y}_{C_2}^TC_2^T\end{bmatrix}.
$$
The regularity of $\lambda E_r-A_r$ follows from the symmetry of $E_r$ and $A_r$ and the fact that 
$\mbox{ker}(E_r)\cap\mbox{ker}(A_r)=\{0\}$. 

\subsection{Stability}
\label{ssec:stability}

Stability is an important physical property of dynamical systems characterizing the sensitivity
of the solution to perturbations in the data. The pencil $\lambda E_r-A_r$ is called {\em stable} if all its finite eigenvalues have non-positive real part, and eigenvalues on the imaginary axis are semi-simple in the sense that they have the same algebraic and geometric multiplicity. In this case, any solution of the DAE system \eqref{eq:MQS3Dreglin} with $u=0$ is bounded.
Furthermore, $\lambda E_r-A_r$ is called {\em asymptotically stable} if all its finite eigenvalues lie in the open left complex half-plane. This implies that any solution of \eqref{eq:MQS3Dreglin} with $u=0$
satisfies $x_r(t)\to 0$ as $t\to\infty$.

The following theorem establishes a~quasi-Weierstrass canonical form for the pencil \mbox{$\lambda E_r-A_r$} which immediately provides information on the finite spectrum and index of this pencil. 

\begin{theorem}\label{th:WCF}
	Let the matrices $E_r$, $A_r\in\mathbb{R}^{n_r\times n_r}$ be as in \eqref{eq:MQS3Dreglinmat}. Then there exists a~nonsingular matrix $W\in\mathbb{R}^{n_r\times n_r}$ which transforms the pencil $\lambda E_r-A_r$ into the quasi-Weierstrass canonical form
	\begin{equation}
	W^TE_rW= \begin{bmatrix} E_{11} &\enskip &\enskip \\ &\enskip I_{n_0} &\enskip \\ &\enskip &\enskip 0\;\end{bmatrix}, \qquad
	W^TA_rW= \begin{bmatrix} A_{11} &\enskip &\enskip \\ &\enskip 0 &\enskip \\ &\enskip &\enskip I_{n_\infty}\end{bmatrix},
	\label{eq:WCF}
	\end{equation}
	where $E_{11}$, $-A_{11}\in\mathbb{R}^{n_s\times n_s}$ are symmetric, positive definite, and 
	\mbox{$n_s+n_0+n_\infty=n_r$}. Furthermore, the pencil $\lambda E_r-A_r$ has index one and all its finite eigenvalues are real and non-positive.
\end{theorem}

\begin{proof}
First, note that the existence of a~nonsingular matrix $W$ transforming $\lambda E_r-A_r$ into \eqref{eq:WCF} immediately follows from the general results for Hermitian pencils  \cite{Tho76}. However, here, we present a~constructive proof to better understand the structural properties of the pencil $\lambda E_r-A_r$.

	Let the columns of the matrices $Y_\sigma\in\mathbb{R}^{n_r\times\, n_\infty}$ and $Y_\nu\in\mathbb{R}^{n_r\times\, n_0}$ form bases of $\mbox{ker}(F_\sigma^T)$ and $\mbox{ker}(F_\nu^T)$, respectively. Then we have 
	\begin{equation}
	F_\sigma^TY_\sigma^{}=0, \qquad F_\nu^TY_\nu^{}=0.
	\label{eq:kernels}
	\end{equation}
	Moreover, the matrices $Y_\nu^TE_r^{}Y_\nu^{}$ and $Y_\sigma^TA_r^{}Y_\sigma^{}$ are both nonsingular,
	and $\left[Y_\nu,\; Y_\sigma\right]$ has full column rank. These properties follow from the fact that 
	$$
	\mbox{ker}(F_\sigma^T)\cap\mbox{ker}(F_\nu^T)=\mbox{ker}(E_r)\cap\mbox{ker}(A_r) = \{0\}.
	$$
	Consider a matrix
	\begin{equation}
	W = \begin{bmatrix} W_1, &\enskip Y_\nu^{}(Y_\nu^TE_r^{}Y_\nu^{})^{-1/2}, & \enskip Y_\sigma^{}(Y_\sigma^TA_r^{}Y_\sigma^{})^{-1/2}\end{bmatrix},
	\label{eq:W}
	\end{equation}
	where the columns of $W_1$ form a~basis of $\mbox{ker}\bigl([E_rY_\nu,\; A_rY_\sigma]^T\bigr)$. First, we show that this matrix is nonsingular. Assume that there exists a~vector $v$ such that $W^Tv=0$. Then 
	$W_1^Tv=0$, $Y_\nu^Tv=0$ and $Y_\sigma^Tv=0$. Thus, 
	$$
	v\in\mbox{im}\bigl([E_rY_\nu,\;A_rY_\sigma]\bigr)\cap\mbox{ker}(Y_\nu^T)\cap\mbox{ker}(Y_\sigma^T)=\{0\},
	$$
	and, hence, $W$ is nonsingular. 
	
	Furthermore, using \eqref{eq:kernels} and 
	$$
	W_1^TE_rY_\nu^{}(Y_\nu^TE_r^{}Y_\nu^{})^{-1/2}=0, \qquad W_1^TA_rY_\sigma^{}(Y_\sigma^TA_r^{}Y_\sigma^{})^{-1/2}=0,
	$$
	we obtain \eqref{eq:WCF} with $E_{11}=W_1^TE_rW_1^{}$ and $A_{11}=W_1^TA_rW_1^{}$. Obviously, $E_{11}$ and $-A_{11}$ are symmetric and positive semidefinite. For any $v_1\in\mbox{ker}(E_{11})$, we have
	$F_\sigma^T W_1v_1=0$. This implies $W_1v_1\in\mbox{ker}(F_\sigma^T)=\mbox{im}(Y_\sigma)$. Therefore,
	there exists a~vector $z$ such that $W_1v_1=Y_\sigma z$. Multiplying this equation from the left with 
	$Y_\sigma^TE_r$, we obtain $Y_\sigma^TE_rY_\sigma^{} z=Y_\sigma^TE_rW_1v_1=0$. Then $z=0$ and, hence, $v_1=0$. Thus, $E_{11}$ is positive definite. Analogously, we can show that $-A_{11}$ is positive definite too. This implies that all eigenvalues of the pencil $\lambda E_{11}-A_{11}$ are real and negative. Index one property immediately follows from \eqref{eq:WCF}.
\end{proof}

 As a~consequence, we obtain that the DAE system \eqref{eq:MQS3Dreglin} is stable but not asymptotically stable since the pencil $\lambda E_r-A_r$ has zero eigenvalues.  
 
We consider now the output equation \eqref{eq:MQS3Dreglinout}. Our goal is to transform this equation to the standard form 
$y=C_rx_r$ with an~output matrix $C_r\in \mathbb{R}^{m\times\, n_r}$. For this purpose, we introduce first a~reflexive inverse of $E_r$ given by
\begin{equation}\label{eq:3DEpinv}
 E_r^- = W\begin{bmatrix} E_{11}^{-1} & &\\&\enskip I&\\&&\enskip 0
                                        \end{bmatrix} W^T.
\end{equation}
Simple calculations show that this matrix satisfies
\begin{align}
E_r^{}E_r^-E_r^{}=E_r^{}, \qquad 
E_r^-E_r^{}E_r^- = E_r^-, \qquad
(E_r^-)^T =E_r^-.
\label{eq:3DEinvT}
  \end{align}
Next, we show that $\hat{Y}_{C_2}^T X_2^{}$ has full column rank. 
Indeed, if there exists a~vector $v$ such that $\hat{Y}_{C_2}^T X_2^{} v=0$, then 
$X_2v\in \ker(\hat{Y}_{C_2}^T)$. On the other hand, 
$$
X_2v=C_2^T\mathit{\Upsilon} v\in \mbox{im}(C_2^T)=\mbox{im}(\hat{Y}_{C_2})
$$ 
implying $X_2v=0$. Since $X_2$ has full column rank, we get $v=0$. 

Using nonsingularity of $X_2^T\hat{Y}_{C_2}^{}\hat{Y}_{C_2}^TX_2^{}$, the input matrix $B_r$ in 
\eqref{eq:MQS3Dreglinmat} can be represented as 
\begin{align}
  B_r&=F_\sigma M_\sigma\begin{bmatrix} 0 \\ I\end{bmatrix}=
	F_\sigma M_\sigma \begin{bmatrix} I & 0\\ X_1^T& X_2^T\hat{Y}_{C_2}\end{bmatrix}
\begin{bmatrix}
 0\\
 \hat{Y}_{C_2}^T X_2^{} (X_2^T\hat{Y}_{C_2}^{}\hat{Y}_{C_2}^T X_2^{})^{-1}
\end{bmatrix}= E_r\begin{bmatrix} 0 \\ Z\end{bmatrix}\label{eq:MQS3DBr2}
\end{align}
  with $Z=\hat{Y}_{C_2}^T X_2^{} (X_2^T\hat{Y}_{C_2}^{}\hat{Y}_{C_2}^T X_2^{})^{-1}$.
  Then employing the first relation in \eqref{eq:3DEinvT} and the state equation \eqref{eq:MQS3Dreglin}, the output \eqref{eq:MQS3Dreglinout} can be written as 
\begin{align*}
 y& =-\begin{bmatrix} 0,\; Z^T\end{bmatrix} E_r\dot{x}_r+R^{-1}u
=-\begin{bmatrix} 0,\; Z^T\end{bmatrix}E_r^{}E_r^-E_r^{}\dot{x}_r+R^{-1}u\\
 &=-B_r^TE_r^-(A_r x_r+B_ru)+R^{-1}u =-B_r^TE_r^-A_r x_r+(R^{-1}-B_r^TE_r^-B_r^{})u.
\end{align*}
It follows from the first relation in \eqref{eq:3DEinvT} and \eqref{eq:MQS3DBr2} that 
\begin{align*}
 B_r^TE_r^-B_r^{}&=\begin{bmatrix} 0,\; Z^T\end{bmatrix}E_r^{}E_r^-E_r^{}
\begin{bmatrix} 0 \\ Z\end{bmatrix}
 =\begin{bmatrix} 0,\; Z^T\end{bmatrix}F_\sigma^{} M_\sigma^{} F_\sigma^T
\begin{bmatrix} 0 \\ Z\end{bmatrix}
 =R^{-1}.
\end{align*}
Thus, the output takes the form 
\begin{equation}
y=C_rx_r
\label{eq:output}
\end{equation} 
with $C_r=-B_r^TE_r^-A_r^{}$. 

\subsection{Passivity}
\label{ssec:passivity}

Passivity is another crucial property of control systems especially in interconnected network design
\cite{AndVon1973,WillTak07}. The DAE control system \eqref{eq:MQS3Dreglin}, \eqref{eq:output} is called 
{\em passive} if for all $t_f>0$ and all inputs $u\in L_2(0,t_f)$ admissible with the initial condition 
$E_r x_r(0)=0$, the output satisfies 
$$
\int_0^{t_f} y^T(t)\,u(t)\, {\rm d}t\geq 0.
$$ 	
This inequality means that the system does not produce energy. In the frequency domain, passivity of  \eqref{eq:MQS3Dreglin}, \eqref{eq:output} is equivalent to the {\em positive definiteness} of its transfer function 
$$
H_r(s)=C_r(sE_r-A_r)^{-1}B_r
$$ 
meaning that $H_r(s)$ is analytic in $\mathbb{C}_+=\{z\in\mathbb{C}\; :\; \mbox{Re}(z)>0\}$ 
and $H_r^{}(s)+H_r^*(s)\geq 0$ for all $s\in\mathbb{C}_+$, see \cite{AndVon1973}. Using the special structure of the system matrices in \eqref{eq:MQS3Dreglinmat}, we can show that 
the DAE system \eqref{eq:MQS3Dreglin}, \eqref{eq:output} is passive.

\begin{theorem}\label{th:passivity}
The DAE system \eqref{eq:MQS3Dreglin}, \eqref{eq:MQS3Dreglinmat}, \eqref{eq:output} is passive.
\end{theorem}

\begin{proof}
First, observe that the transfer function $H_r(s)$ of \eqref{eq:MQS3Dreglin}, \eqref{eq:MQS3Dreglinmat}, \eqref{eq:output} is
analytic on $\mathbb{C}_+$. This fact immediately follows from Theorem~\ref{th:WCF}. 
Furthermore, introducing the function
\mbox{$F(s)=(sE_r -A_r )^{-1}B_r$} and using the relations
$$
E_r^{}E_r^-A_r^{}=E_r^{}E_r^-A_r^{}E_r^-E_r^{}=A_r^{}E_r^-E_r^{},
$$
we obtain
\begin{equation*}
\arraycolsep=2pt
\begin{array}{rcl}
H_r^{}(s)+H_r^*(s)&=& C_r^{} (sE_r^{} -A_r^{} )^{-1}B_r^{} +B_r^T(\overline{s}E_r^{} -A_r^{} )^{-1}C_r^T\\
&=&-B_r^TE_r^- A_r^{} (sE_r^{} -A_r^{} )^{-1}B_r^{} -B_r^T(\overline{s}E_r^{} -A_r^{})^{-1}A_r^{} E_r^- B_r^{} \\
&=&F^*(s)\bigl(-(\overline{s}E_r^{} -A_r^{} )E_r^- A_r^{} -A_r^{} E_r^- (sE_r^{} -A_r^{} )\bigr)F(s)\\
&=&2\,F^*(s)\bigl(A_r^{} E_r^- A_r^{} +\mbox{\rm Re}(s)E_r^{} E_r^- (-A_r^{} )E_r^- E_r^{}\bigr)F(s)\geq 0
\end{array}
\end{equation*}
for all $s\in \mathbb{C}_+$. In the last inequality, we utilized the property that 
the matrices $A_r^{} E_r^- A_r^{} $ and $E_r^{} E_r^- (-A_r^{} )E_r^- E_r^{} $ are both symmetric and positive semidefinite. Thus, $H_r(s)$ is positive real, and, hence, system \eqref{eq:MQS3Dreglin}, \eqref{eq:MQS3Dreglinmat}, \eqref{eq:output} is passive.
\end{proof}

\section{Balanced Truncation Model Reduction}

Our goal is now to approximate the DAE system \eqref{eq:MQS3Dreglin}, \eqref{eq:MQS3Dreglinmat}, \eqref{eq:output} by a~reduced-order model 
\begin{equation}
\label{eq:MQS3Dred}
\arraycolsep2pt
\begin{array}{rcl}
\tilde{E}_r\dot{\tilde{x}}_r&=&\tilde{A}_r\tilde{x}_r+\tilde{B}_ru,\\
\tilde{y}&=&\tilde{C}_r \tilde{x}_r,
\end{array}
\end{equation}
where $\tilde{E}_r$, $\tilde{A}_r\in\mathbb{R}^{\ell\times \ell}$, $\tilde{B}_r$, $\tilde{C}_r^T\in\mathbb{R}^{\ell\times m}$ and $\ell\ll n_r$. This model should capture the dynamical behavior 
of \eqref{eq:MQS3Dreglin}. It is also important that it preserves the passivity and has a~small approximation error. In order to determine the reduced-order model \eqref{eq:MQS3Dred}, we aim to employ a~balanced truncation model reduction method \cite{Antoulas2005, Moore81}. 
Unfortunately, we cannot apply this method directly to \eqref{eq:MQS3Dreglin}, \eqref{eq:MQS3Dreglinmat}, \eqref{eq:output} because, as established in 
Section~\ref{ssec:stability}, this system is stable but not asymptotically stable 
due to the fact that the pencil $\lambda E_r -A_r $ has zero eigenvalues. Another difficulty is the presence of infinite eigenvalues due to the singularity of $E_r$. This may cause problems in defining the controllability and observability Gramians which play an essential role in balanced truncation. 

To overcome these difficulties, we first observe that the states of the transformed system 
$(W^TE_rW, W^TA_rW, W^TB_r, C_rW)$ corresponding to the zero and infinite eigenvalues
 are uncontrollable and unobservable at the same time. This immediately follows from the representations
\begin{equation}
W^TB_r=[B_1^T,\; 0, \; 0]^T, \qquad C_rW=[C_1,\; 0, \; 0].
\label{eq:WtB}
\end{equation} 
with $B_1=W_1^TB_r$ and $C_1=-B_r^TE_r^{-}A_r^{}W_1^{}=-B_1^TE_{11}^{-1}A_{11}^{}$.
Therefore, these states can be removed from the system without changing its input-output beha\-vior. Then the standard balanced truncation approach can be applied to the remaining system. Since the system matrices of the regularized system 
\eqref{eq:MQS3Dreglin},~\eqref{eq:output} have the same structure as those of RC circuit equations
studied in \cite{ReiSty2009}, we proceed with the balanced truncation approach developed there which avoids the computation of the transformation matrix~$W$.

For the DAE system \eqref{eq:MQS3Dreglin},~\eqref{eq:output}, we define the controllability and obser\-va\-bi\-li\-ty Gramians $G_c$ and $G_o$ as unique symmetric, positive semidefinite solutions of the projected continuous-time Lyapunov equations 
\begin{align}
E_r G_cA_r +A_r G_cE_r  &=-\mathit{\Pi}^TB_r B_r ^T\mathit{\Pi}, \quad G_c=\mathit{\Pi} G_c\mathit{\Pi}^T, \label{eq:MQSprojLyapcont} \\
E_r G_oA_r +A_r G_oE_r  &=-\mathit{\Pi}^TC_r ^TC_r \mathit{\Pi}, \quad G_o=\mathit{\Pi} G_o\mathit{\Pi}^T, \label{eq:MQSprojLyapobs}
\end{align}
where $\mathit{\Pi}$ %=I-\mathit{\Pi}_0-\mathit{\Pi}_\infty$ 
is the spectral projector onto the right deflating subspace of $\lambda E_r -A_r $ corresponding to the negative eigenvalues. 
Using the quasi-Weierstrass canonical form \eqref{eq:WCF} and \eqref{eq:W}, this projector can be represented as 
\begin{equation}
\mathit{\Pi} = W\begin{bmatrix} I &\enskip\enskip & \\ &\enskip 0 \enskip& \\ 
& \enskip\enskip& 0\end{bmatrix}W^{-1} = W_1^{}\hat{W}_1^T,
\label{eq:Pi}
\end{equation}
where $\hat{W}_1\in\mathbb{R}^{n_r\times\,n_s}$ satisfies
\begin{equation}
\hat{W}_1^T W_1^{} = I,\qquad \hat{W}_1^TY_\nu=0,\qquad \hat{W}_1^TY_\sigma=0.
\label{eq:hatW1}
\end{equation}
Similarly to \cite[Theorem~3]{KS17}, a~relation between the controllability and the obser\-va\-bi\-li\-ty Gramians of system \eqref{eq:MQS3Dreglin}, \eqref{eq:MQS3Dreglinmat},~\eqref{eq:output} can be established. %The only difference is that we take the pseudoinverse of $E_r$ and $A_r$. 

\begin{theorem}\label{lem:EQE3D} 
		Let $G_c$ and $G_o$ be the controllability and observability Gramians of system \eqref{eq:MQS3Dreglin}, \eqref{eq:MQS3Dreglinmat},~\eqref{eq:output} which solve the projected Lyapunov equations 	\eqref{eq:MQSprojLyapcont} and \eqref{eq:MQSprojLyapobs}, respectively. Then 
	\begin{equation*}
	E_r G_oE_r =A_r G_cA_r.
	\end{equation*}
\end{theorem}

\begin{proof}
	Consider the reflexive inverse $E_r^-$  of $E_r$ given in \eqref{eq:3DEpinv} and the reflexive inverse of $A_r$ given by
	\begin{equation*}
	A_r^- =W\begin{bmatrix}
	A_{11}^{-1} & & \\ &\enskip 0 &\\ & & \enskip I
	\end{bmatrix}W^T.
	\end{equation*}
	Then multiplying the Lyapunov equation \eqref{eq:MQSprojLyapcont} (resp. \eqref{eq:MQSprojLyapobs}) from the left and right with $E_r^-$  (resp. with $A_r^-$) and using the relations 
	\begin{equation*}
	\arraycolsep2pt
	\begin{array}{rclrclrcl}
	E_r\mathit{\Pi} & = & \mathit{\Pi}^TE_r, 
&\qquad \mathit{\Pi} E_r^-  &=&E_r^- \mathit{\Pi}^T, 
& \mathit{\Pi}^TE_r^{}E_r^- & = & \mathit{\Pi}^TA_r^{}A_r^-,\\
	A_r\mathit{\Pi} & = & \mathit{\Pi}^TA_r, 
& \qquad\mathit{\Pi} A_r^- &=&A_r^- \mathit{\Pi}^T,
&\qquad E_r^-A_r^{}A_r^- & = & E_r^-\mathit{\Pi}^T,
	\end{array}
	\end{equation*}
	we obtain
	\begin{align}
	A_r^- (A_rG_cA_r)E_r^- +E_r^- (A_rG_cA_r)A_r^- &=- \mathit{\Pi} E_r^- B_r^{} B_r ^TE_r^- \mathit{\Pi}^T, \enskip
	G_c=\mathit{\Pi} G_c\mathit{\Pi}^T, \label{eq:MQSprojLyapconmult} \\
	A_r^- (E_rG_oE_r)E_r^- +E_r^- (E_rG_oE_r)A_r^- &=- \mathit{\Pi}E_r^-B_r^{}B_r^TE_r^- \mathit{\Pi}^T, \enskip G_o=\mathit{\Pi} G_o\mathit{\Pi}^T.
	\label{eq:MQSprojLyapobsmult}
		\end{align}
Since $E_r^-$ and $-A_r^-$ are symmetric and positive semidefinite and $\mathit{\Pi}^T$ is the spectral projector onto the right deflating subspace of $\lambda E_r^- -A_r^- $ corresponding to the negative eigenvalues, the Lyapunov equations \eqref{eq:MQSprojLyapconmult} and \eqref{eq:MQSprojLyapobsmult} are uniquely solvable, and, hence, $E_r G_o E_r=A_rG_cA_r$.
\end{proof}

Theorem \ref{lem:EQE3D} implies that we need to solve only the projected Lyapunov equation~\eqref{eq:MQSprojLyapcont} for the Cholesky factor $Z_c$ of $G_c=Z_c^{}Z_c^T$. 
Then it follows from the relation 
\begin{equation*}
G_o=E_r^- A_rG_cA_rE_r^- =(-E_r^- A_rZ_c)(-Z_c^TA_rE_r^- )
\end{equation*} 
that
the Cholesky factor of the observability Gramian $G_o=Z_o^{}Z_o^T$ can be calculated as \mbox{$Z_o=-E_r^- A_rZ_c$}.
In this case, the Hankel singular values of \eqref{eq:MQS3Dreglin},~\eqref{eq:output} can be computed from the eigenvalue decomposition 
\begin{equation*}
Z_o^TE_r^{}Z_c^{}=(-Z_c^TA_r^{}E_r^- )E_r^{}Z_c^{}=-Z_c^TA_r^{}Z_c^{}=\begin{bmatrix}
U_1,\; U_2
\end{bmatrix}\begin{bmatrix}\Lambda_1 & \\& \Lambda_2
\end{bmatrix}\begin{bmatrix}U_1,\; U_2
\end{bmatrix}^T,
\end{equation*}
where $\begin{bmatrix} U_1,\; U_2 \end{bmatrix}$ is orthogonal, $\Lambda_1=\mbox{diag}(\lambda_1,\ldots,\lambda_{\ell})$ and $\Lambda_2=\mbox{diag}(\lambda_{\ell+1},\ldots,\lambda_{n_r})$
with  $\lambda_1\geq \ldots\geq\lambda_{\ell}\gg\lambda_{\ell+1}\geq\ldots\geq\lambda_{n_r}$.
Then the reduced-order model \eqref{eq:MQS3Dred} is computed by projection 
\begin{equation*}
\tilde{E}_r=W^TE_rV,\qquad \tilde{A}_r=W^TA_rV,\qquad \tilde{B}_r=W^TB_r ,\qquad \tilde{C}_r=C_r V
\end{equation*}
with the projection matrices $V=Z_c^{}U_1^{}\Lambda_1^{-\frac12}$ and $W=Z_o^{}U_1^{}\Lambda_1^{-\frac12}=-E_r^- A_rV$.
The reduced matrices have the form
\begin{align}
\tilde{E}_r=&-V^TA_r^{}E_r^- E_r^{}V=-\Lambda_1^{-\frac12}U_1^TZ_c^TA_r^{}Z_c^{}U_1^{}\Lambda_1^{-\frac12}=I,\nonumber\\
\tilde{A}_r=&-V^TA_r^{}E_r^- A_r^{}V,\label{eq:redmatr}\\
\tilde{B}_r=&-V^TA_r^{}E_r^- B_r^{} =V^TC_r ^T=\tilde{C}_r^T.\nonumber
\end{align}
The balanced truncation method for the DAE system \eqref{eq:MQS3Dreglin}, \eqref{eq:MQS3Dreglinmat}, \eqref{eq:output} is presented in Algorithm~\ref{algo:BTMQS3D}, where for numerical efficiency reasons, the Cholesky factor $Z_c$ of the Gramian $G_c$ is replaced by a low-rank Cholesky factor $\tilde{Z}_c$ such that $G_c\approx\tilde{Z}_c^{}\tilde{Z}_c^T$ .

\begin{algorithm}[ht]
	\caption{Balanced truncation for the 3D linear MQS system} \label{algo:BTMQS3D}
	\begin{algorithmic}[1]
		\REQUIRE 
		$E_r$, $A_r\in\mathbb{R}^{n_r\times\,n_r}$ and $B_r\in\mathbb{R}^{n_r\times\, m}$
		\ENSURE a reduced-order system $(\tilde{E}_r, \tilde{A}_r, \tilde{B}_r, \tilde{C}_r)$.
		\STATE 	Solve the projected Lyapunov equation \eqref{eq:MQSprojLyapcont}
		for a~low-rank Cholesky factor $\tilde{Z}_c\in\mathbb{R}^{n_r\times n_c}$ of the controllability Gramian $G_c\approx\tilde{Z}_c^{}\tilde{Z}_c^T$.
		\STATE 	Compute the eigenvalue decomposition
		\begin{equation*}
		-\tilde{Z}_c^TA_r\tilde{Z}_c=\begin{bmatrix} U_1,\; U_2\end{bmatrix}\begin{bmatrix}
		\Lambda_1 &0\\0&\Lambda_2
		\end{bmatrix}\begin{bmatrix} U_1,\; U_2\end{bmatrix}^T,
		\end{equation*}
		where $\begin{bmatrix} U_1,\; U_2 \end{bmatrix}$ is orthogonal, $\Lambda_1=\mbox{diag}(\lambda_1,\ldots,\lambda_{\ell})$ and $\Lambda_2=\mbox{diag}(\lambda_{\ell+1},\ldots,\lambda_{n_c})$.
		\STATE 	Compute the reduced matrices 
		$$
		\tilde{E}_r=I, \quad\tilde{A}_r=-V^TA_rE_r^- A_r V, \quad 
		\tilde{B}_r=-V^TA_rE_r^- B_r, \quad \tilde{C}_r=\tilde{B}_r^T
		$$
		with the projection matrix $V=\tilde{Z}_c^{}U_1^{}\Lambda_1^{-\frac12}$.
	\end{algorithmic}
\end{algorithm}

Note that the matrices $\tilde{E}_r$ and $-\tilde{A}_r$ in \eqref{eq:redmatr} are both symmetric and positive definite. This implies that the reduced-order model  \eqref{eq:MQS3Dred}, \eqref{eq:redmatr} is asymptotically stable and its transfer function $\tilde{H}_r(s)=\tilde{C}_r(s\tilde{E}_r-\tilde{A}_r)^{-1}\tilde{B}_r$ satisfies 
$$
\arraycolsep=2pt
\begin{array}{rcl}
\tilde{H}_r^{}(s)+\tilde{H}_r^*(s) & = & \tilde{B}_r^T(s\tilde{E}_r^{}-\tilde{A}_r^{})^{-1}\tilde{B}_r^{}+ \tilde{B}_r^T(\overline{s}\tilde{E}_r^{}-\tilde{A}_r^{})^{-1}\tilde{B}_r \\
& = & 2\tilde{B}_r^T(\overline{s}\tilde{E}_r^{}-\tilde{A}_r^{})^{-1}\bigl(\mbox{Re}(s)\tilde{E}_r^{}-\tilde{A}_r^{}\bigr)(s\tilde{E}_r^{}-\tilde{A}_r^{})^{-1}\tilde{B}_r^{}\geq 0
\end{array}
$$ 
for all $s\in\mathbb{C}_+$. Thus, $\tilde{H}_r(s)$ is positive real and, hence, the reduced-order model~\eqref{eq:MQS3Dred} is passive. Moreover, taking into account that the controllability and obser\-vability Gramians $\tilde{G}_c$ and $\tilde{G}_o$ of \eqref{eq:MQS3Dred} satisfy $\tilde{G}_c=\tilde{G}_o=\Lambda_1>0$, we conclude that \eqref{eq:MQS3Dred} is balanced and minimal. 
Finally, we obtain the following bound on the $\mathcal{H}_\infty$-norm of the approximation error
\begin{equation}
\|H_r-\tilde{H}_r\|_{\mathcal{H}_\infty}:=\sup_{\omega\in\mathbb{R}}\|H_r(i\omega)-\tilde{H}_r(i\omega)\| \leq 2(\lambda_{\ell+1}+\ldots+\lambda_{n_r}),
\label{eq:bound}
\end{equation}
which can be proved analogously to \cite{Enns1984,Glover1984}. Note that using \eqref{eq:WCF} and \eqref{eq:WtB}, the error system can be written as
\begin{align*}
H_r(s)-\tilde{H}_r(s) & = C_r(sE_r-A_r)^{-1}B_r - \tilde{C}_r(s\tilde{E}_r-\tilde{A}_r)^{-1}\tilde{B}_r\\
& = B_1^T\bigl(sE_{11}^{}(-A_{11}^{-1})E_{11}^{}-(-E_{11}^{})\bigr)^{-1}B_1^{} - \tilde{B}_r^T(s\tilde{E}_r^{}-\tilde{A}_r^{})^{-1}\tilde{B}_r^{}\\
& = C_e(sE_e-A_e)^{-1}B_e
\end{align*}
with 
$$
E_e=\begin{bmatrix} -E_{11}^{}A_{11}^{-1}E_{11}^{} & \\ & \tilde{E}_r\end{bmatrix}, \qquad
A_e=\begin{bmatrix} -E_{11}^{} & \\ & \tilde{A}_r\end{bmatrix}, \qquad
B_e=\begin{bmatrix} B_1 \\ \tilde{B}_r\end{bmatrix}=C_e^T.
$$
Since $E_e$ and $-A_e$ are both symmetric, positive definite and $B_e^{}=C_e^T$, it follows from \cite[Theorem~4.1(iv)]{ReiSty2009} that $\|H_r-\tilde{H}_r\|_{\mathcal{H}_\infty}=\|H_r(0)-\tilde{H}_r(0)\|$.
Using the output equation \eqref{eq:MQS3Dreglinout} instead of \eqref{eq:output}, the transfer function $H_r(s)$ can also be written as 
$$
H_r(s)=-sB_r^T(sE_r^{}-A_r^{})^{-1}B_r^{}+R^{-1}.
$$ 
Then the computation of the $\mathcal{H}_\infty$-error is simplified to 
\begin{equation}
\|H_r-\tilde{H}_r\|_{\mathcal{H}_\infty}= \|R^{-1}+\tilde{B}_r^T\tilde{A}_r^{-1}\tilde{B}_r^{}\|.
\label{eq:error}
\end{equation}
We will use this relation in numerical experiments to verify the efficiency of the error bound~\eqref{eq:bound}.

\section{Computational Aspects}
\label{sec:numerics}

In this section, we discuss the computational aspects of Algorithm~\ref{algo:BTMQS3D}. This includes solving  
the projected Lyapunov equation \eqref{eq:MQSprojLyapcont} and computing the basis matrices for certain subspaces.

For the numerical solution of the projected Lyapunov equation \eqref{eq:MQSprojLyapcont} in Step~1 of Algorithm~\ref{algo:BTMQS3D}, we apply the low-rank alternating directions implicit (LR-ADI) method as
presented in \cite{Sty2008} with appropriate modifications proposed in \cite{BenKurSaa2013-2} for cheap evaluation of the Lyapunov residuals. First, note that due to \eqref{eq:WtB} the input matrix satisfies 
$\mathit{\Pi}^TB_r=B_r$. Then setting 
\begin{align*}
F_1 & = (\tau_1E_r+A_r)^{-1}B_r, \\
R_1 & = B_r-2\tau_1E_rF_1, \\
Z_1 & = \sqrt{-\tau_1}F_1,
\end{align*}
the LR-ADI iteration is given by
\begin{equation}
\arraycolsep=2pt
\begin{array}{rcl}
F_k & = & (\tau_k E_r+A_r)^{-1} R_{k-1}, \\
R_k & = & R_{k-1}-2\tau_1E_rF_k, \\
Z_k & = & [Z_{k-1}, \; \sqrt{-\tau_k}F_k],
\end{array}
\label{eq:LR-ADI}
\end{equation}
with negative shift parameters $\tau_k$ which strongly influence the convergence of this iteration.
Note that they can be chosen to be real, since the pencil $\lambda E_r-A_r$ has real finite eigenvalues.  This also enables to determine the optimal ADI shift parameters by the Wachspress method
\cite{Wach2009} ones the spectral bounds $a=-\lambda_{\max}(E_r,A_r)$ and \mbox{$b=-\lambda_{\min}(E_r,A_r)$} are available. Here, $\lambda_{\max}(E_r,A_r)$ and $\lambda_{\min}(E_r,A_r)$ denote the largest and smallest nonzero eigenvalues of $\lambda E_r-A_r$. They can be computed simultaneously by applying the Lanczos procedure to $E_r^-A_r^{}$ and $v=\mathit{\Pi}v$, see \cite[Section~10.1]{GoluV13}. As a~starting vector 
$v$, we can take, for example, one of the columns of the  matrix $E_r^-B_r^{}$. In the Lanczos procedure and also in Step~3 of Algorithm~\ref{algo:BTMQS3D}, it is required to compute the products $E_r^-A\mathit{\Pi}v$. Of course, we never compute and store the reflexive inverse $E_r^-$ explicitly. Instead, we can use the following lemma to calculate such products in a~numerically efficient way.                             

\begin{lemma}\label{lem:E+}
	Let $E_r$ and $A_r$ be given as in \eqref{eq:MQS3Dreglinmat}, $Z=\hat{Y}_{C_2}^TX_2^{}(X_2^T\hat{Y}_{C_2}^{}\hat{Y}_{C_2}^TX_2^{})^{-1}$, and $v\in \mathbb{R}^{n_r}$.
	Then the vector $z=E_r^-A_r^{}\mathit{\Pi}v$ can be determined as
	\begin{equation}
	z=(I-\mathit{\Pi}_\infty)\hat{Y}_\sigma^{} (\hat{Y}_\sigma^TE_r^{}\hat{Y}_\sigma^{})^{-1}\hat{Y}_\sigma^TA_r\mathit{\Pi}v, 
\label{eq:E-Av}
	\end{equation}
	where $\mathit{\Pi}_\infty = Y_\sigma^{}(Y_\sigma^TA_r^{}Y_\sigma^{})^{-1}Y_\sigma^TA_r^{}$
	is a~spectral projector onto the right deflating subspace of $\lambda E_r-A_r$ corresponding to the eigenvalue at infinity, and 
	\begin{equation}
	\hat{Y}_\sigma = \begin{bmatrix} I &\enskip 0\\ 0&\enskip Z\end{bmatrix}
	\label{eq:hatYsigma}
	\end{equation}
	is a basis matrix for $\mbox{\rm im}(F_{\sigma})$.
\end{lemma}

\begin{proof}
	We show first that the full column matrix $\hat{Y}_\sigma$ in \eqref{eq:hatYsigma} satisfies
	$\mbox{im}(\hat{Y}_\sigma)=\mbox{\rm im}(F_{\sigma})$. This property immediately follows from 
	the relation 
	\begin{equation*}
	F_{\sigma}=\begin{bmatrix}I&\enskip X_1 \\ 0 & \enskip \hat{Y}_{C_2}^TX_2 \end{bmatrix} = 
	\begin{bmatrix} I &\enskip 0\\0 &\enskip Z\end{bmatrix} 
	\begin{bmatrix} I &\enskip X_1\\ 0 &\enskip X_2^T\hat{Y}_{C_2}^{}\hat{Y}_{C_2}^TX_2^{}\end{bmatrix}.
	\end{equation*}
	 %Then $[Y_\sigma,\; \hat{Y}_\sigma]$ is nonsingular and 
	Since $F_\sigma^T\hat{Y}_\sigma^{}$ has full column rank, the matrix 
	$\hat{Y}_\sigma^TE_r^{}\hat{Y}_\sigma^{}=\hat{Y}_\sigma^TF_\sigma^{}F_\sigma^T\hat{Y}_\sigma^{}$ is nonsingular, i.e., $z$ in \eqref{eq:E-Av} is well-defined. Obviously, this vector fulfills  $\mathit{\Pi}_\infty z=0$. Furthermore, we have 
	$$
	E_rz=E_r(I-\Pi_\infty)\hat{Y}_\sigma^{} (\hat{Y}_\sigma^TE_r^{}\hat{Y}_\sigma^{})^{-1}\hat{Y}_\sigma^TA_r^{}\mathit{\Pi}v=
	E_r\hat{Y}_\sigma^{} (\hat{Y}_\sigma^TE_r^{}\hat{Y}_\sigma^{})^{-1}\hat{Y}_\sigma^TA_r^{}\mathit{\Pi}v.
	$$
	Then 
	\begin{align*}
	\hat{Y}_\sigma^TE_rz & = \hat{Y}_\sigma^TA_r^{}\mathit{\Pi}v,\\
	Y_\sigma^TE_rz & = 0 = Y_\sigma^T(I-\mathit{\Pi}_\infty^T)A_r^{}\mathit{\Pi}v = Y_\sigma^TA_r^{}\mathit{\Pi}v.
	\end{align*}
	Since $[\hat{Y}_\sigma, \,Y_\sigma]$ is nonsingular, these equations imply
	$E_rz=A_r^{}\mathit{\Pi}v$. Multiplying this equation from the left with $E_r^-$, we get
	%	\begin{equation}\label{eq:lemEinvprof3}
	$$
	z=(I-\mathit{\Pi}_{\infty})z= E_r^- E_r^{}z=E_r^- A_r^{}\mathit{\Pi}v.
	%\end{equation}
	$$
	This completes the proof.
\end{proof}

Using \eqref{eq:hatYsigma}, we find by simple calculations that
$$
\hat{Y}_\sigma^{} (\hat{Y}_\sigma^TE_r^{}\hat{Y}_\sigma^{})^{-1}\hat{Y}_\sigma^T = 
\begin{bmatrix} M_{11}^{-1} &\enskip -M_{11}^{-1}X_1^{}Z^T\\[2mm] -ZX_1^TM_{11}^{-1}&\enskip Z(X_1^TM_{11}^{-1}X_1^{}+R)Z^T\end{bmatrix}.
$$
Next, we discuss the computation of $Y_{\sigma}(Y_{\sigma}^TA_rY_{\sigma})^{-1}Y_{\sigma}^Tv$ for a~vector $v$. By taking \mbox{$v=A_rw$}, this enables to calculate the product 
$\mathit{\Pi}_\infty w=Y_{\sigma}^{}(Y_{\sigma}^TA_r^{}Y_{\sigma}^{})^{-1}Y_{\sigma}^TA_r^{}w$ required in 
\eqref{eq:E-Av}. 
	
	\begin{lemma}\label{lem:Piinf}
		Let $A_r$ be as in \eqref{eq:MQS3Dreglinmat} and let $Y_{\sigma}$ be a~basis of $\,\ker(F_{\sigma}^T)$. Then for $v=[v_1^T,\, v_2^T ]^T\in \mathbb{R}^{n_r}$, the product  
		\begin{equation}
		z=Y_{\sigma}^{}(Y_{\sigma}^TA_r^{}Y_{\sigma}^{})^{-1}Y_{\sigma}^Tv
		\label{eq:z}
		\end{equation}
		can be determined as $z=[0,\, z_2^T]^T$, where $z_2$ satisfies the linear system 
		\begin{equation}\label{eq:Piinf}
		\begin{bmatrix}
		-\hat{Y}_{C_2}^TK_{22}\hat{Y}_{C_2}& \hat{Y}_{C_2}^TX_2\\[1mm] X_2^T\hat{Y}_{C_2}&0
		\end{bmatrix}
		\begin{bmatrix}
		z_2 \\[1.2mm] \hat{z}_2
		\end{bmatrix}
		=\begin{bmatrix} v_2 \\[1.2mm] 0 	\end{bmatrix}.
		\end{equation}
	\end{lemma}
	
	\begin{proof}
		We first show that  $z=Y_{\sigma}(Y_{\sigma}^TA_rY_{\sigma})^{-1}Y_{\sigma}^Tv$ if and only if
		\begin{equation}\label{eq:Piinfprof1}
		\begin{bmatrix}
		A_r & \hat{Y}_{\sigma}\\\hat{Y}_{\sigma}^T &0 
		\end{bmatrix}\begin{bmatrix} z \\ \hat z\end{bmatrix}=\begin{bmatrix}v \\0\end{bmatrix},
		\end{equation}
		where $\hat{Y}_{\sigma}$ is as in \eqref{eq:hatYsigma}.
			Let $[z^T,\, \hat{z}^T]^T$ solves equation \eqref{eq:Piinfprof1}. 
		Then $\hat{Y}_\sigma^T z=0$ and, hence, \mbox{$z\in\mbox{ker}(\hat{Y}_\sigma^T)=\mbox{\rm im}(Y_\sigma)$}.
		This means that there exists a~vector $\hat{w}$ such that $z=Y_\sigma \hat{w}$. Inserting this vector into the first equation in \eqref{eq:Piinfprof1}, we obtain $A_rY_\sigma \hat{w} + \hat{Y}_\sigma \hat{z}=v$.
		Multiplying this equation from the left with $Y_\sigma^T$ and solving it for $\hat{w}$, we get
		$z=Y_\sigma^{}(Y_\sigma^TA_r^{} Y_\sigma^{})^{-1}Y_\sigma^T v$.	
		
		Conversely, for $z$ as in \eqref{eq:z} and
		$\hat{z}= (\hat{Y}_{\sigma}^T\hat{Y}_{\sigma}^{})^{-1}\hat{Y}_{\sigma}^T(v-A_r^{}z)$,
		we have $\hat{Y}_{\sigma}^T z=0$ and
		\begin{align*}
		A_r^{} z+\hat{Y}_{\sigma}^{}\hat{z} &= 
		A_r^{} z + \hat{Y}_{\sigma}^{}(\hat{Y}_{\sigma}^T\hat{Y}_{\sigma}^{})^{-1}\hat{Y}_{\sigma}^T(v-A_r^{}z)\\
		&=(I-\hat{Y}_{\sigma}^{}(\hat{Y}_{\sigma}^T\hat{Y}_{\sigma}^{})^{-1}\hat{Y}_{\sigma}^T)A_r^{}
		z+\hat{Y}_{\sigma}^{}(\hat{Y}_{\sigma}^T\hat{Y}_{\sigma}^{})^{-1}\hat{Y}_{\sigma}^T v.
		\end{align*}
		Using $\hat{Y}_{\sigma}^{}(\hat{Y}_{\sigma}^T\hat{Y}_{\sigma}^{})^{-1}\hat{Y}_{\sigma}^T+
		Y_{\sigma}^{}(Y_{\sigma}^TY_{\sigma}^{})^{-1}Y_{\sigma}^T=I$ twice, we obtain
		\begin{align*}
		A_r^{} z+\hat{Y}_{\sigma}^{}\hat{z}&=Y_{\sigma}^{}(Y_{\sigma}^TY_{\sigma}^{})^{-1}
		Y_{\sigma}^TA_r^{}z+
		\hat{Y}_{\sigma}^{}(\hat{Y}_{\sigma}^T\hat{Y}_{\sigma}^{})^{-1}\hat{Y}_{\sigma}^Tv\\
		&=Y_{\sigma}^{}(Y_{\sigma}^TY_{\sigma}^{})^{-1}Y_{\sigma}^TA_r^{}Y_{\sigma}^{}
		(Y_{\sigma}^TA_r^{}Y_{\sigma}^{})^{-1}Y_{\sigma}^Tv+\hat{Y}_{\sigma}^{}
		(\hat{Y}_{\sigma}^T\hat{Y}_{\sigma}^{})^{-1}\hat{Y}_{\sigma}^Tv = v.
		\end{align*}
		Thus, $[z^T,\, \hat{z}^T]^T$ satisfies equation \eqref{eq:Piinfprof1}.
		
		Equation \eqref{eq:Piinfprof1} can be written as
		\begin{equation}\label{eq:Piinfprof2}
		\begin{bmatrix}
		-K_{11} &  -K_{12}\hat{Y}_{C_2}&\enskip I &\enskip 0\\
		- \hat{Y}_{C_2}^TK_{21} & - \hat{Y}_{C_2}^TK_{22}\hat{Y}_{C_2}&\enskip 0 &\enskip  Z\\
		I & 0 &\enskip  0 &\enskip  0\\ 
		0& Z^T & \enskip 0 &\enskip  0\;
		\end{bmatrix}\begin{bmatrix} z_1 \\ z_2\\ z_3\\ z_4\end{bmatrix}=
		\begin{bmatrix} v_1\\v_2\\0\\0
		\end{bmatrix},
		\end{equation}
		with $z=[z_1^T,\, z_2^T]^T$, $\hat{z}=[z_3^T,\, z_4^T]^T$ and $v=[v_1^T,\, v_2^T]^T$.
		The third equation in \eqref{eq:Piinfprof2} yields $z_1=0$. Furthermore, multiplying the fourth equation in  \eqref{eq:Piinfprof2} from the left with $X_2^T\hat{Y}_{C_2}^{}\hat{Y}_{C_2}^TX_2^{}$ and introducing a~new variable $\hat{z}_2=(X_2^T\hat{Y}_{C_2}^{}\hat{Y}_{C_2}^TX_2^{})^{-1}z_4$, we obtain equation \eqref{eq:Piinf} which is uniquely solvable since 
		$\hat{Y}_{C_2}^TK_{22}\hat{Y}_{C_2}$ is symmetric, positive definite and $\hat{Y}_{C_2}^TX_2^{}$ has full column rank. Thus, $z =[0,\, z_2^T]^T$ with $z_2$ satisfying \eqref{eq:Piinf}.
	\end{proof}

We summarize the computation of $z=E_r^-A_r^{}v$ with $v=\mathit{\Pi}v$ in Algorithm~\ref{alg:invEA}.

\begin{algorithm}[ht]
	\caption{Computation of $E_r^-A_r^{}v$} \label{alg:invEA}
	\begin{algorithmic}[1]
		\REQUIRE 
		$M_{11}$, $K_{11}$,
		$K_{12}$, 
		$K_{21}$, 
		$K_{22}$,
		$X_1$,
		$X_2$, 
		$R$, 
		$\hat{Y}_{C_2}$,
				and $v=\mathit{\Pi}v=[v_1^T,\,v_2^T]^T$. 
			\ENSURE $z=E_r^-A_r^{}v$ with $E_r$ and $A_r$ as in \eqref{eq:MQS3Dreglinmat}.
		\STATE 	Compute $\displaystyle{\begin{bmatrix}\hat{v}_1 \\ \hat{v}_2\end{bmatrix}=
			\begin{bmatrix} -K_{11} v_1-K_{12}\hat{Y}_{C_2}v_2 \\ -\hat{Y}_{C_2}^TK_{21}^{}v_1^{}-\hat{Y}_{C_2}^TK_{22}^{}\hat{Y}_{C_2}^{} v_2^{}
			\end{bmatrix}}$.
		\STATE 	Compute $Z=\hat{Y}_{C_2}^TX_2^{}(X_2^T\hat{Y}_{C_2}^{}\hat{Y}_{C_2}^TX_2^{})^{-1}$.
		\STATE 	Compute $\hat{w}_2=Z^T\hat{v}_2$.
		\STATE 	Solve $M_{11} w_1=\hat{v}_1-X_1\hat{w}_2$ for $w_1$.
		\STATE 	Compute $w_2=-Z(X_1^Tw_1-R\hat{w}_2)$.
		\STATE  Solve $\displaystyle{\begin{bmatrix} -\hat{Y}_{C_2}^TK_{22}^{}\hat{Y}_{C_2}^{} & \hat{Y}_{C_2}^TX_2^{} \\ X_2^T\hat{Y}_{C_2}^{} & 0\end{bmatrix}
			\begin{bmatrix} z_2 \\ \hat{z}_2\end{bmatrix}=
			\begin{bmatrix}-\hat{Y}_{C_2}^TK_{21}^{}w_1^{}-\hat{Y}_{C_2}^TK_{22}^{}\hat{Y}_{C_2}^{} w_2^{} \\ 0\end{bmatrix}}$ for $z_2$.
		\STATE Compute $\displaystyle{z=\begin{bmatrix} w_1 \\ w_2-z_2\end{bmatrix}}$.
	\end{algorithmic}
\end{algorithm}

The major computational effort in the LR-ADI method \eqref{eq:LR-ADI} is the computation of 
\mbox{$(\tau_k E_r+A_r)^{-1}w$} for some vector $w$. 
If $\tau_k E_r+A_r$ remains sparse, we just solve the linear system $(\tau_k E_r+A_r)z=w$ of dimension $n_r$. 
If $\tau_k E_r+A_r$ gets
fill-in due to the multiplication with $\hat{Y}_{C_2}$, then we can use the following lemma to compute $z=(\tau_k E_r+A_r)^{-1}w$.

\begin{lemma}
	Let $E_r$ and $A_r$ be as in \eqref{eq:MQS3Dreglinmat}, $w=[w_1^T,\,w_2^T]^T\in\mathbb{R}^{n_r}$, and $\tau<0$. Then the vector $z=(\tau E_r+A_r)^{-1}w$ can be determined as 
	$$
	z=\begin{bmatrix} z_1 \\ (\hat{Y}_{C_2}^T\hat{Y}_{C_2}^{})^{-1}\hat{Y}_{C_2}^T z_2\end{bmatrix},
	$$ 
	where $z_1$ and $z_2$ satisfy the linear system
	\begin{equation}
	\label{eq:MQS3DADIeq}
	\begin{bmatrix}
	\tau M_{11}-K_{11} & -K_{12} & X_1 & 0\\
	-K_{21} & -K_{22} & X_2 & Y_{C_2}\\[.5mm]
	\tau X_1^T & \tau X_2^T & -R & 0\\[.5mm]
	0 & Y_{C_2}^T & 0 & 0
	\end{bmatrix}
	\begin{bmatrix}
	z_1\\z_2\\z_3\\z_4
	\end{bmatrix}=
	\begin{bmatrix}
	w_1 \\ \hat{Y}_{C_2}^{}(\hat{Y}_{C_2}^T\hat{Y}_{C_2}^{})^{-1}w_2\\0\\0
	\end{bmatrix}
	\end{equation}
	of dimension $n+m+k_2$.
\end{lemma}

\begin{proof}
	First, note that due to the choice of $Y_{C_2}$ the coefficient matrix in system~\eqref{eq:MQS3DADIeq} is nonsingular. This system can be written as
	\begin{subequations}
		\begin{alignat}{8}
		(\tau M_{11}-K_{11})z_1 && -K_{12}z_2 && +X_1z_3 && && &= w_1, \label{eq:MQS3DADI1}\\
		-K_{21}z_1 && -K_{22}z_2 && +X_2z_3 && +Y_{C_2}z_4 && &=\hat{Y}_{C_2}^T(\hat{Y}_{C_2}^T\hat{Y}_{C_2})^{-1}w_2,\label{eq:MQS3DADI2}\\
		\tau X_1^Tz_1 &&+\tau X_2^Tz_2 && -Rz_3 && && &=0,\label{eq:MQS3DADI3}\\
		&&Y_{C_2}^Tz_2 && && && &=0.\label{eq:MQS3DADI4}
		\end{alignat}
	\end{subequations}
	It follows from \eqref{eq:MQS3DADI4} that $z_2\in \ker(Y_{C_2}^T)=\mbox{\rm im}(\hat{Y}_{C_2})$. Then there exists $\hat z_2$ such that $z_2=\hat{Y}_{C_2}\hat z_2$. Since $\hat{Y}_{C_2}$ has full column rank,
	it holds 
		\begin{equation}\label{eq:MQS3DADI5}
	\hat z_2=(\hat{Y}_{C_2}^T\hat{Y}_{C_2}^{})^{-1}\hat{Y}_{C_2}^Tz_2.
	\end{equation}
	Further, from equation \eqref{eq:MQS3DADI3} we obtain $z_3=\tau R^{-1}X_1^T z_1 +\tau R^{-1}X_2^T z_2$. 
	Substituting $z_2$ and $z_3$ into \eqref{eq:MQS3DADI1} and \eqref{eq:MQS3DADI2} and multiplying equation \eqref{eq:MQS3DADI2} from the left with $\hat{Y}_{C_2}^T$ yields
	\begin{equation*}
	(\tau E_r+A_r)\begin{bmatrix}
	z_1\\\hat z_2
	\end{bmatrix}=\begin{bmatrix} w_1\\w_2\end{bmatrix}.
	\end{equation*}
	This equation together with \eqref{eq:MQS3DADI5} implies that 
	\begin{equation*}
	\begin{bmatrix}
	z_1\\(\hat{Y}_{C_2}^T\hat{Y}_{C_2}^{})^{-1}\hat{Y}_{C_2}^T z_2
	\end{bmatrix}=(\tau E_r+A_r)^{-1}\begin{bmatrix}
	w_1\\w_2\end{bmatrix}
	\end{equation*}
	that  completes the proof.
\end{proof}

Finally, we discuss the computation of the basis matrices $Y_{C_2}$ and $\hat{Y}_{C_2}$ 
required in Algorithm~\ref{alg:invEA} and the LR-ADI iteration. 
To this end, we introduce a~{\em discrete gradient matrix} $G_0\in\mathbb{R}^{n_e\times n_n}$ whose entries 
are defined as
\begin{equation*}
(G_0)_{ij}=\begin{cases}
  \phantom{-}1, & \text { if edge $i$ leaves node $j$},\\
-1, & \text { if edge $i$ enters node $j$},\\
\phantom{-}0, &\text{ else.}
\end{cases}
\end{equation*}
Note that the discrete curl and gradient matrices $C$ and $G_0$ satisfy 
\mbox{$\mbox{rank}(C)=n_e-n_n+1$}, $\mbox{rank}(G_0)=n_n-1$ and $CG_0=0$, see \cite{Boss1998}. 
Then by removing one column of $G_0$, we get the reduced discrete gradient matrix $G$  
whose columns form a~basis of $\mbox{ker}(C)$. The matrices $C$ and $G^T$ can be considered as the loop and
 incidence matrices, respectively, of a~directed graph
whose nodes and branches correspond to the nodes and edges of the triangulation 
$\mathcal{T}_h(\Omega)$, 
see \cite{Deo74}. Then the basis matrices $Y_{C_2}$ and $\hat{Y}_{C_2}$ can be determined by using the graph-theoretic algorithms as presented in \cite{Ipac2013}.

 Let the reduced gradient matrix $G=\begin{bmatrix} G_1^T & G_2^T\end{bmatrix}^T$
be partitioned into blocks according to $C=\begin{bmatrix}C_1, \; C_2\end{bmatrix}$.
It follows from \cite[Theorem 9]{Ipac2013} that
\begin{equation*}
 \mbox{ker}(C_2)=\mbox{im}(G_2 Z_1),
\end{equation*}
where the columns of the matrix $Z_1$ form a~basis of $\ker(G_1)$. Then $\hat{Y}_{C_2}$ can be determined as $\hat{Y}_{C_2}=\texttt{kernelAk}(Z_1^TG_2^T)$ with the function $\texttt{kernelAk}$ from \cite[Section 4.2]{Ipac2013}, where the basis $Z_1$ is computed by applying the function $\texttt{kernelAT}$ from \cite[Section 3]{Ipac2013} to $G_1^T$.

\section{Numerical Results}

In this section,  we present some results of numerical experiments demonstrating the balanced truncation model reduction method for 3D linear MQS systems. 
For the FEM discretization with N\'ed\'elec elements, we used the 3D tetrahedral mesh generator NETGEN\footnote{https://sourceforge.net/projects/netgen-mesher/} and 
the MATLAB toolbox\footnote{http://www.mathworks.com/matlabcentral/fileexchange/46635} from  \cite{AnjaVald2015} for assembling the system matrices. 
All computations were done with MATLAB~R2018a. 

As a test model, we consider a~coil wound round  a~conducting tube surrounded by air. Such a~model was studied in \cite{NicST14} in the context of optimal control problems. A~bounded domain  
$$
\Omega = (-c_1, c_1)\times (-c_2, c_2)\times (-c_3, c_3)\subset\mathbb{R}^3 
$$
consists of the conducting domain $\Omega_1 = \Omega_{\rm iron}$ of the iron tube 
and the non-conducting domain $\Omega_2 = \Omega_{\rm coil}\cup\Omega_{\rm air}$,
where
$$
\arraycolsep=2pt
\begin{array}{rcl}
\Omega_{\rm iron} & = & \{\xi\in\mathbb{R}^3\enskip :\enskip 0<r_1<\xi_1^2+\xi_2^2<r_2,\enskip  z_1<\xi_3<z_2\;\},\\[2mm]
\Omega_{\rm coil} & = & \{\xi\in\mathbb{R}^3\enskip :\enskip 0<r_3<\xi_1^2+\xi_2^2<r_4,\enskip  z_3<\xi_3<z_4\;\}
\end{array}
$$
with $r_1<r_2<r_3<r_4$ and $z_1<z_3<z_4<z_2$, see Fig.~\ref{fig:tube}(a). The dimensions, geometry and material parameters are given in Fig.~\ref{fig:tube}(b). 
The divergence free winding function $\chi:\Omega\to\mathbb{R}^3$ is defined by
$$
\chi(\xi)=\left\{ \begin{array}{ll} \displaystyle{\frac{N_c}{S_c\sqrt{\xi_1^2+\xi_2^2}}}\begin{bmatrix}-\xi_2 \\ \enskip\; \xi_1 \\ \enskip\;0\end{bmatrix}, & \quad\xi\in\Omega_{\rm coil}, \\[3mm] \qquad 0, & \quad\xi\in\Omega\setminus\Omega_{\rm coil}, \end{array}\right.
$$
where $N_c$ is the number of coil turns and $S_c$ is the cross section area of the coil.

\begin{figure}[t]
\begin{minipage}{0.4\textwidth}
\framebox{\begin{minipage}{.93\linewidth}\includegraphics[width=\linewidth]{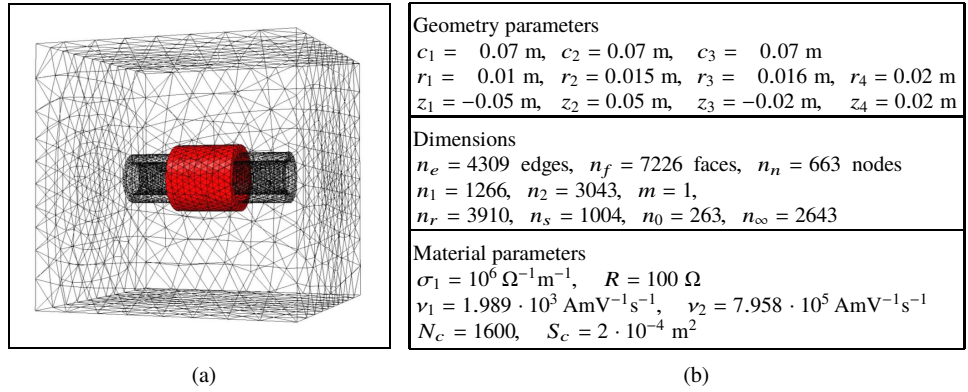}
\end{minipage}}
\end{minipage}
\scalebox{.96}{\begin{minipage}{0.6\textwidth}
	\centering	
	\begin{tabular}{|l|}
		\hline 
		\\[-2mm]
		Geometry parameters \\
		$\begin{array}{llll} 
		c_1=\enskip\; 0.07\; {\rm m}, &\; c_2=0.07\; {\rm m}, &\; c_3=\enskip\; 0.07\; {\rm m} &\\
		r_1=\enskip\; 0.01\; {\rm m}, &\; r_2=0.015\; {\rm m}, &\; r_3=\enskip\; 0.016\; {\rm m},      &\; r_4=0.02\; {\rm m} \\
		z_1=-0.05\; {\rm m},   &\; z_2=0.05\; {\rm m},   &\; z_3=-0.02\; {\rm m},      &\; z_4=0.02\; {\rm m}
		\end{array}$ \\[1mm] \hline
		\\[-2mm] 
		Dimensions \\
		$\begin{array}{l}
		n_e = 4309\enskip \mbox{edges}, \enskip n_f = 7226\enskip \mbox{faces}, \enskip n_n = 663\enskip \mbox{nodes}  \\
		%n t = 3579 \mbox{tetrahedra}, 
		n_1=1266, \enskip n_2=3043, \enskip  m=1, \\
		n_r=3910, \enskip n_s=1004, \enskip n_0=263, \enskip n_\infty=2643 \\
		\end{array}$ \\[1mm] \hline
		\\[-2mm] 
		Material parameters \\
		$\begin{array}{l}
		\sigma_1=10^6\,\Omega^{-1}{\rm m}^{-1}, \quad 
		R=%\displaystyle{\frac{N_c}{S_c\,\sigma_1}}=
		100\;\Omega\\
			\nu_1=1.989\cdot 10^3\,{\rm A}{\rm m}{\rm V}^{-1}{\rm s}^{-1}, \quad
			%14872\,{\rm A}{\rm m}{\rm V}^{-1}{\rm s}^{-1}\\
			\nu_2=7.958\cdot 10^5\,{\rm A}{\rm m}{\rm V}^{-1}{\rm s}^{-1}\\
		N_c=1600,\quad  S_c=2\cdot 10^{-4}\; {\rm m}^2\\[.4mm]
		\end{array}$ \\[.8mm] \hline 
	\end{tabular}\\[.2mm]
\end{minipage} 
}\\[2mm]
\begin{minipage}{0.4\linewidth}\centering (a)\end{minipage}  \hfill
\begin{minipage}{.6\linewidth}\centering (b)\end{minipage} 
\caption{Coil-tube model: (a) geometry; (b) dimensions and model parameters.}
\label{fig:tube}
\end{figure}

\begin{figure}[h]
	\begin{minipage}{0.5\textwidth}
	\includegraphics[width=1.05\linewidth]{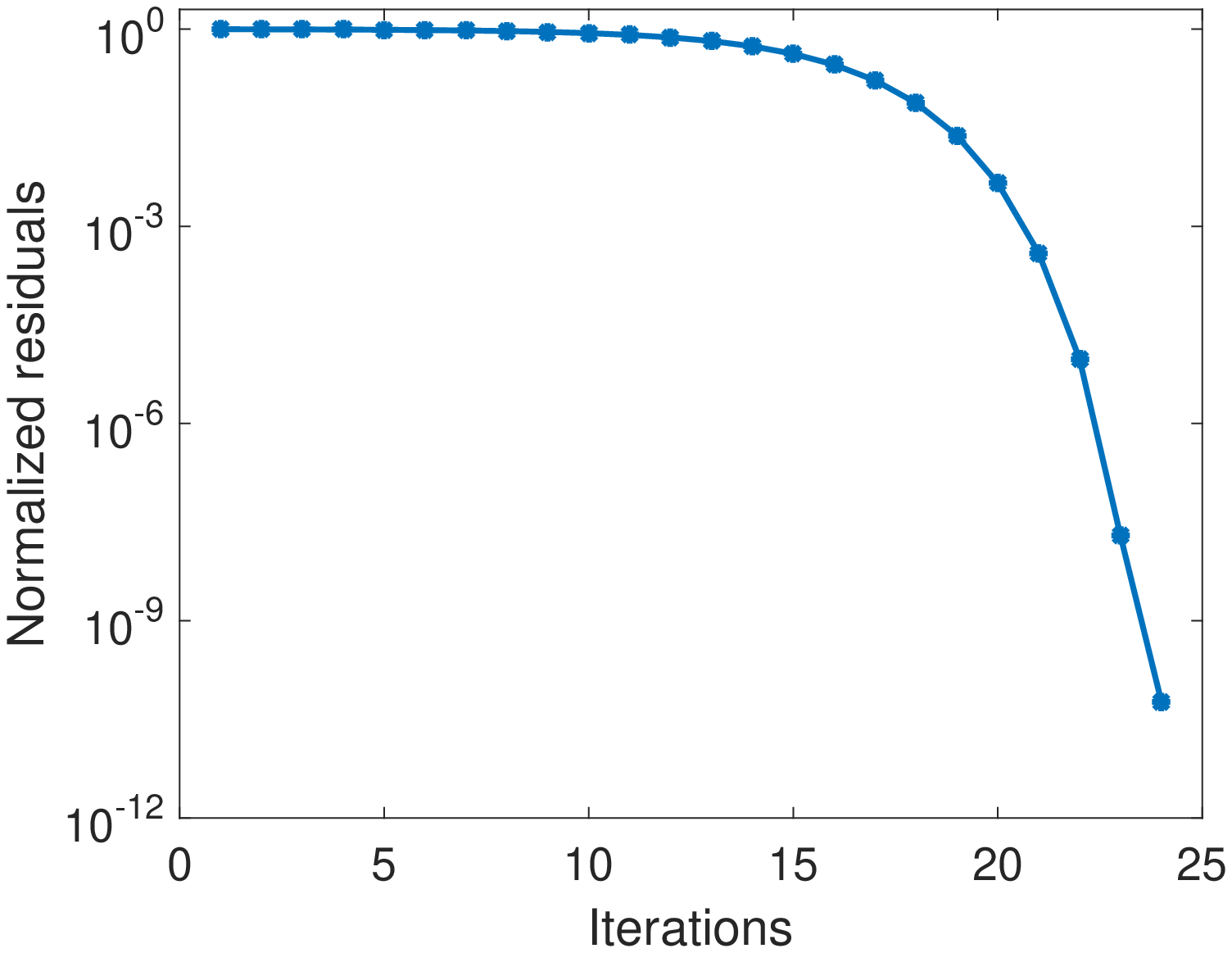}
\end{minipage}
\begin{minipage}{0.5\textwidth} 
	\includegraphics[width=1.05\linewidth]{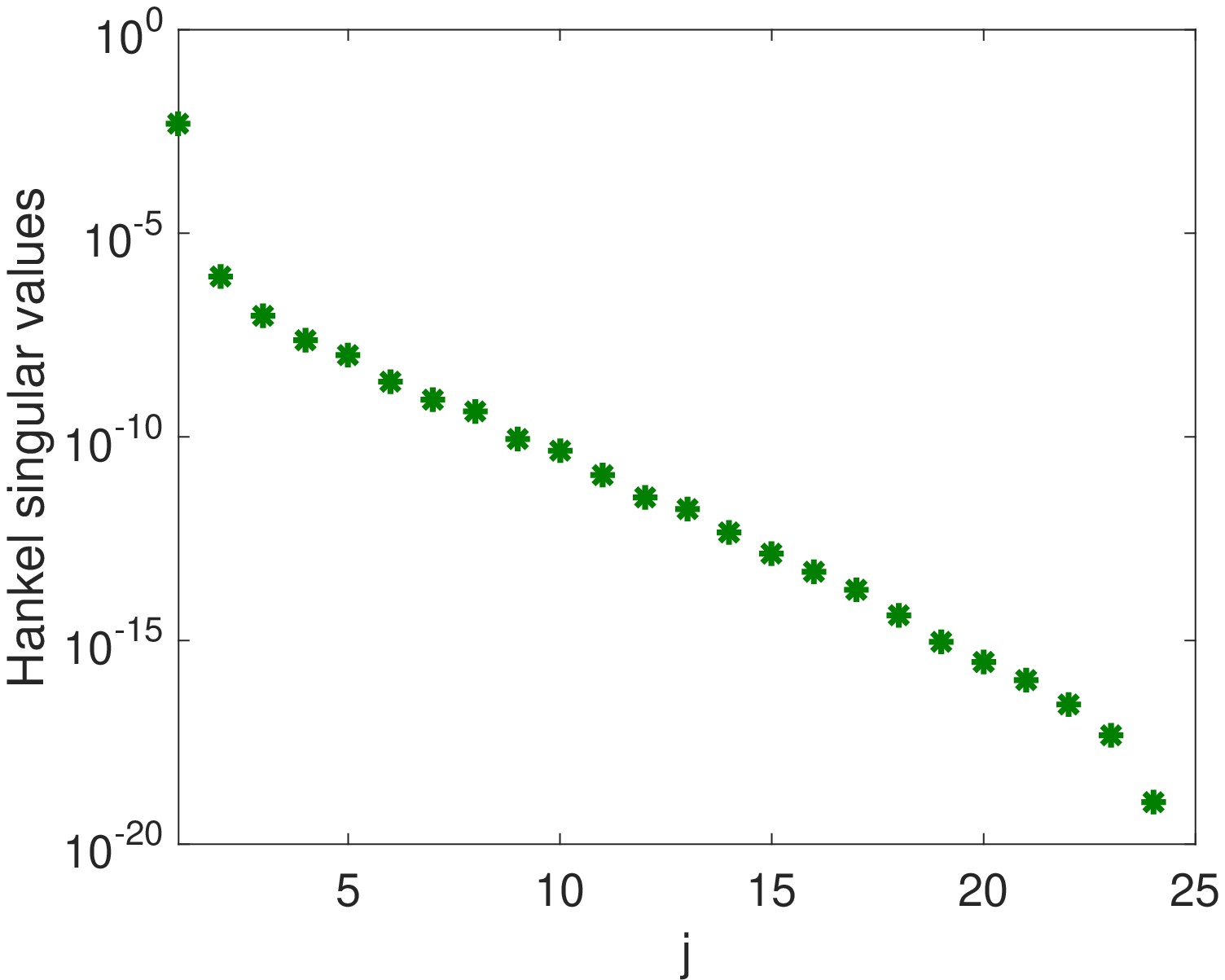}
\end{minipage} \\[2mm]
\begin{minipage}{.5\linewidth}\centering (a)\end{minipage}  \hfill
\begin{minipage}{.5\linewidth}\centering (b)\end{minipage} 
\caption{(a) Convergence history for the LR-ADI method; (b) Hankel singular values.}	
	\label{fig:adi}
\end{figure}

\begin{figure}[ht]
	\begin{minipage}{0.5\textwidth}
		\includegraphics[width=1.05\linewidth]{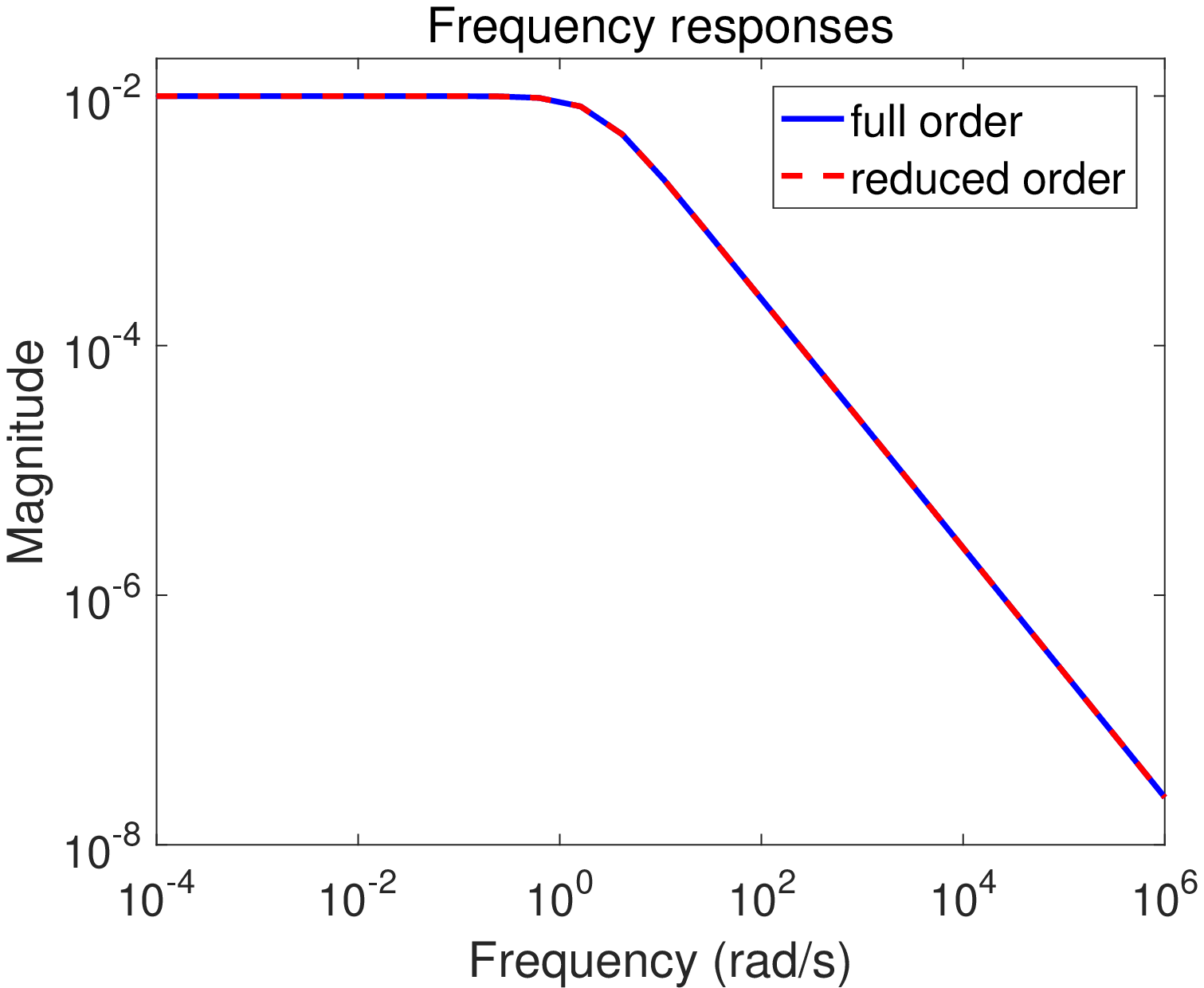}
	\end{minipage} 
\begin{minipage}{0.5\textwidth}
	\includegraphics[width=1.05\linewidth]{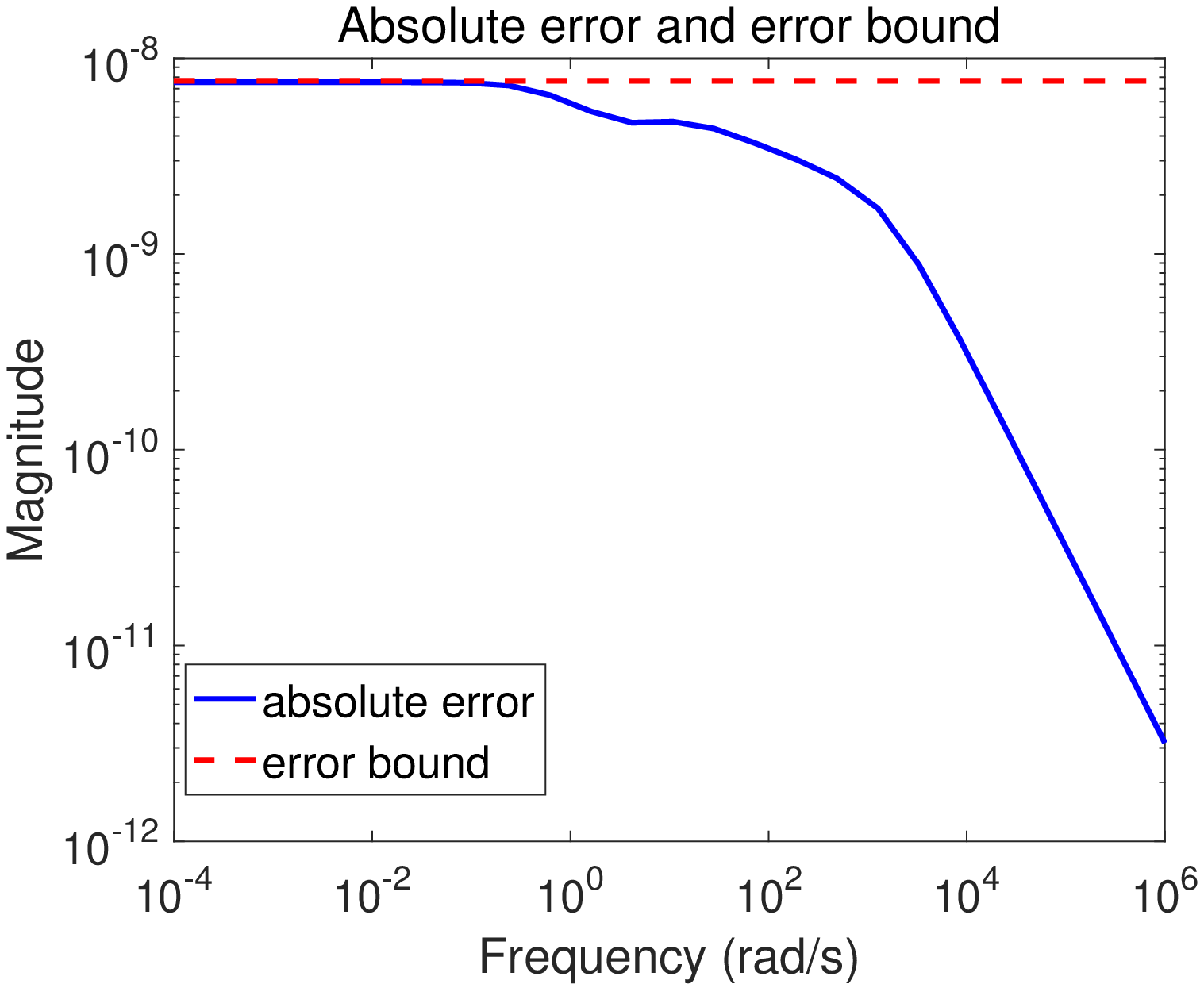}
\end{minipage} 
	\\[2mm]
	\begin{minipage}{.5\linewidth}\centering (a)\end{minipage}  \hfill
	\begin{minipage}{.5\linewidth}\centering (b)\end{minipage} 
	\caption{(a) Frequency responses of the full-order and reduced-order systems; 
		(b) Absolute error and error bound. }	
	\label{fig:frresp}
\end{figure}

\begin{figure}[ht]
	\begin{minipage}{0.5\textwidth}
		\includegraphics[width=1.05\linewidth]{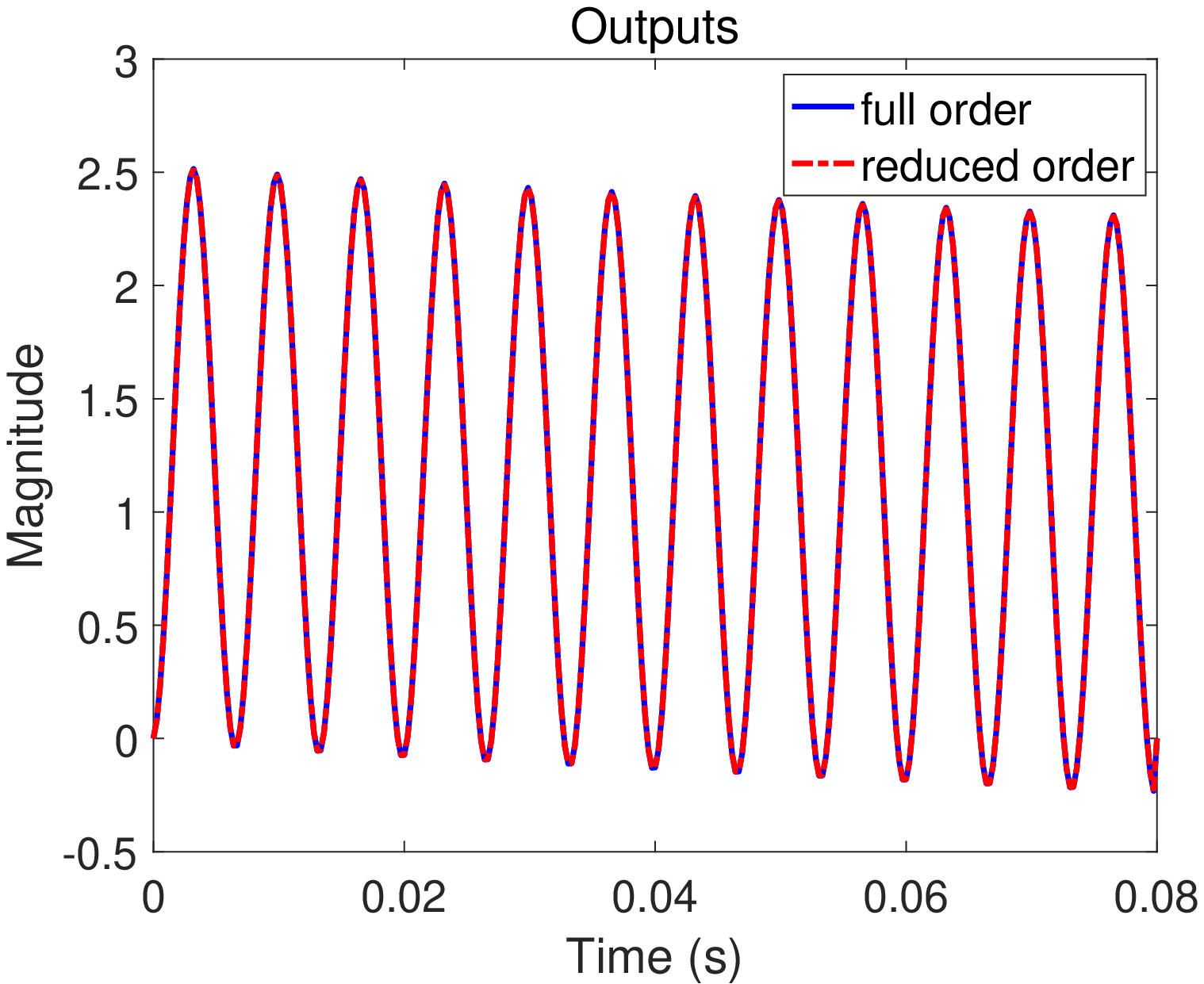}
	\end{minipage} 
    \begin{minipage}{0.5\textwidth}
	    \includegraphics[width=1.05\linewidth]{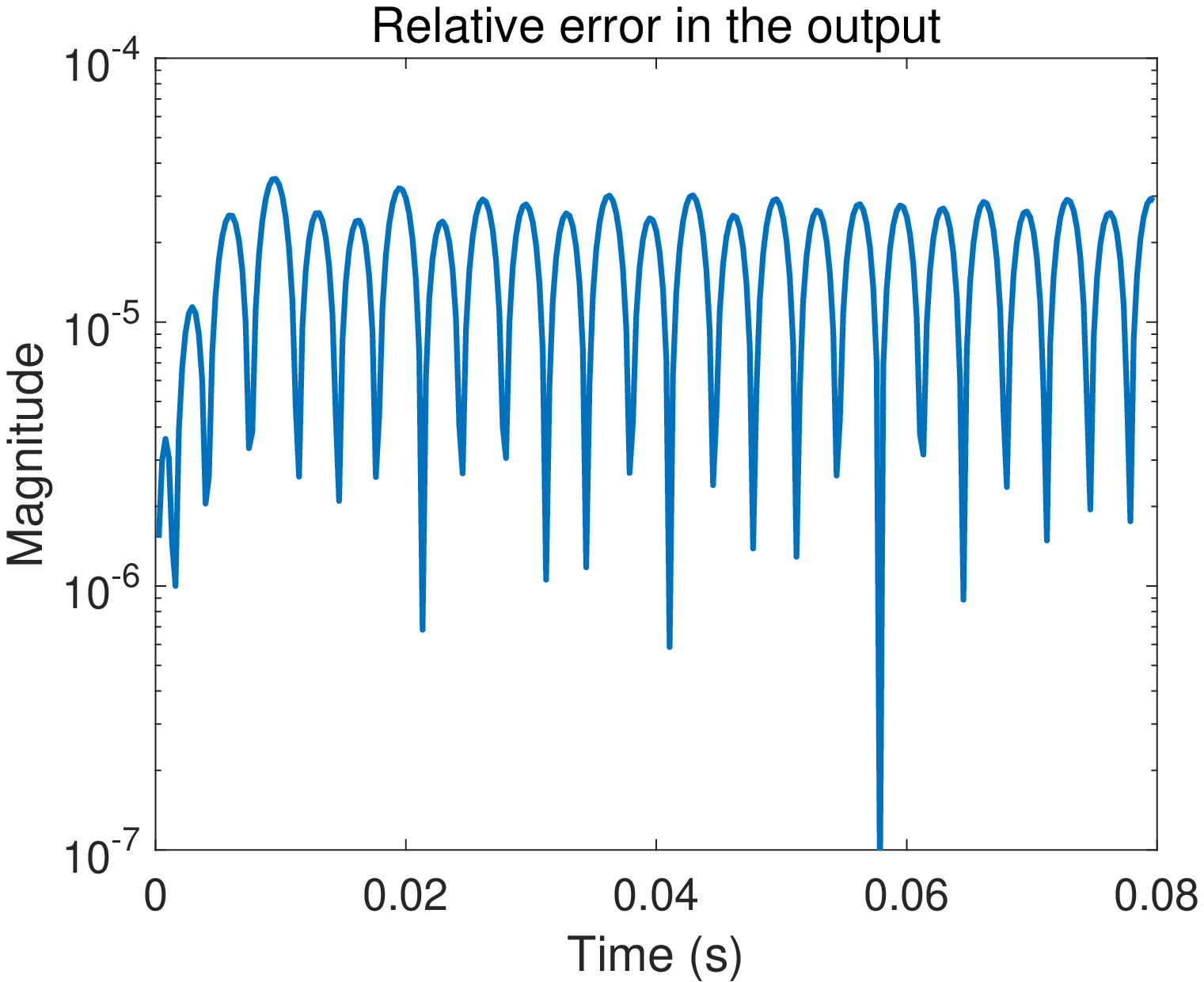}
    \end{minipage} 
	\\[2mm]
	\begin{minipage}{.5\linewidth}\centering (a)\end{minipage}  \hfill
	\begin{minipage}{.5\linewidth}\centering (b)\end{minipage} 
	\caption{(a) Outputs of the full-order and reduced-order systems; 
		(b) Relative error in the output. }	
	\label{fig:time}
\end{figure}

The controllability Gramian was approximated by a low-rank matrix \mbox{$G_c\approx Z_{n_c}^{}Z_{n_c}^T$} with $Z_{n_c}\in\mathbb{R}^{n_r\times\, n_c}$ with $n_c=24$. 
The normalized residual norm 
$$
\frac{\|E_r^{}Z_k^{}Z_k^TA_r^{}+A_r^{}Z_k^{}Z_k^TE_r^{}+B_r^{}B_r^T\|_F}{\|B_r^{}B_r^T\|_F} = \frac{\|R_k^{}R_k^T\|_F}{\|B_r^{}B_r^T\|_F}=\frac{\|R_k^TR_k^{}\|_F}{\|B_r^TB_r^{}\|_F}
$$ 
for the LR-ADI iteration \eqref{eq:LR-ADI} is presented in Fig.~\ref{fig:adi}(a). Fig.~\ref{fig:adi}(b) shows the Hankel singular values $\lambda_1,\ldots,\lambda_{n_c}$. We approximate the regularized MQS system \eqref{eq:MQS3Dreglin}, \eqref{eq:MQS3Dreglinout}
of dimension $n_r=3910$ by a~reduced model of dimension $\ell=5$. In Fig.~\ref{fig:frresp}(a), we present
the absolute values of the frequency responses $|H_r(i\omega)|$ and $|\tilde{H}_r(i\omega)|$ of the full and reduced-order models for the frequency range $\omega\in[10^{-4}, 10^6]$.
The absolute error $|H_r(i\omega)-\tilde{H}_r(i\omega)|$ and the error bound computed as
$$
2\bigl(\lambda_{\ell+1}+\ldots+\lambda_{n_c-1}+(n_s-\ell+1)\lambda_{n_c}\bigr) = 7.6714\cdot 10^{-9}
$$
are given in Fig.~\ref{fig:frresp}(b). Furthermore, using \eqref{eq:error} we compute the error 
$$
\|H_r-\tilde{H}_r\|_{\mathcal{H}_\infty}=7.5385\cdot 10^{-9}
$$
showing that the error bound is very tight.

In Fig.~\ref{fig:time}(a), we present the outputs $y(t)$ and $\tilde{y}(t)$ of the full and reduced-order systems on the time interval $[0,0.08]s$ computed for the input $u(t)=5\cdot 10^4 \sin(300\pi t)$
and zero initial condition using the implicit Euler method with $300$ time steps. The relative error 
$$
\frac{|y(t)-\tilde{y}(t)|}{\max\limits_{t\in[0,0.08]}|y(t)|}
$$ is given in 
Fig.~\ref{fig:time}(b). One can see that the reduced-order model approximates well the original system in both time and frequency domain.

\bigskip\noindent
{\bf Acknowledgment:} The authors would like to thank Hanko Ipach for providing the MATLAB functions for 
computing the kernels and ranges of incidence  matrices.

\bibliographystyle{spmpsci}
%\bibliography{mybib}

\end{document}